\documentclass{amsart}
\usepackage{amssymb}
\usepackage{latexsym}
\usepackage{euscript}

   \def\sH{{\mathfrak H}}   
   \def\sK{{\mathfrak K}}   \def\sL{{\mathfrak L}}
\def\sM{{\mathfrak M}}   \def\sN{{\mathfrak N}}

\def\st{{\mathfrak t}}
\def\ss{{\mathfrak s}}

\def\bB{{\mathbf{B}}}
\def\bE{{E}}

      \def\dC{{\mathbb C}}

   \def\dN{{\mathbb N}}   
      \def\dR{{\mathbb R}}

\def\cG{{\EuScript G}}

   \def\cW{{\EuScript W}}

\def\wt#1{{{\widetilde #1} }}
\def\h#1{{{\hat #1} }}
\def\wh#1{{{\widehat #1} }}

\def\bm\chi{\mbox{\boldmath$\chi$}}

\def\RE{{\rm Re\,}}

\def\ker{{\rm ker\,}}
\def\ran{{\rm ran\,}}
\def\cran{{\rm \overline{ran}\,}}
\def\dom{{\rm dom\,}}
\def\mul{{\rm mul\,}}

\def\cdom{{\rm \overline{dom}\,}}

\def\col{{\rm col\,}}

\let\xker=\ker \def\ker{{\xker\,}}

\def\uphar{{\upharpoonright\,}}

\DeclareMathOperator{\hplus}{\, \widehat + \,}
\DeclareMathOperator{\hoplus}{\, \widehat \oplus \,}

\newtheorem{theorem}{Theorem}[section]
\newtheorem{proposition}[theorem]{Proposition}
\newtheorem{corollary}[theorem]{Corollary}
\newtheorem{lemma}[theorem]{Lemma}

\theoremstyle{definition}

\newtheorem{remark}[theorem]{Remark}
\newtheorem{definition}[theorem]{Definition}

\numberwithin{equation}{section}

\title[Selfadjoint extensions ]
{Selfadjoint extensions of relations whose domain and range are orthogonal} 
\author{S.~Hassi}
\author{J.-Ph.~Labrousse}
\author{H.S.V.~de~Snoo}

\address{Department of Mathematics and Statistics \\
University of Vaasa \\
P.O. Box 700, 65101 Vaasa \\
Finland}
\email{sha@uwasa.fi}

\address{63 Avenue Cap de Croix, 06100 Nice, France}
\email{labro@unice.fr}

\address{Bernoulli Institute for Mathematics, Computer Science and Artificial Intelligence \\
 University of Groningen \\
P.O. Box 407, 9700 AK Groningen \\
Nederland}
\email{hsvdesnoo@gmail.com}

\subjclass[2010]{Primary 47A06, 47B25; Secondary 47A12, 47B65}

\dedicatory{Dedicated to our friend Yury Arlinski\u{\i} on the occasion of his
seventieth birthday}

\keywords{Symmetric operator, nonnegative operator, linear relation, selfadjoint extension, extremal extension, numerical range, boundary triplet, Weyl function}

\begin{document}

\begin{abstract}
The selfadjoint extensions of a closed linear relation $R$ from
a Hilbert space $\sH_1$ to a Hilbert space $\sH_2$ are considered in the Hilbert space
$\sH_1\oplus\sH_2$ that contains the graph of $R$.
They will be described by $2 \times 2$ blocks of linear relations
and by means of boundary triplets associated with a closed symmetric
relation $S$ in $\sH_1 \oplus \sH_2$ that is induced by $R$.
Such a relation is characterized by the orthogonality property $\dom S \perp \ran S$
and it is nonnegative. All nonnegative selfadjoint extensions $A$,
in particular the Friedrichs and Kre\u{\i}n-von Neumann extensions, are parametrized via
an explicit block formula. In particular, it is shown that $A$ belongs to
the class of extremal extensions of $S$ if and only if $\dom A \perp \ran A$.
In addition, using asymptotic properties of an associated Weyl function,
it is shown that there is a natural correspondence between semibounded selfadjoint extensions of $S$
and semibounded parameters describing them if and only if the operator part of $R$ is bounded.
 \end{abstract}

\maketitle

\section{Introduction}

Let $R$ be a closed linear relation from  a Hilbert space $\sH_1$ to
a Hilbert space $\sH_2$. The problem considered here
is to construct selfadjoint relations
that extend the relation $R$ in the larger Hilbert space $\sH_1\oplus\sH_2$.
Then,
based on the case that $R$ is a densely defined closed operator,
one expects that the block of linear relations
\begin{equation}\label{symmkk}
 K=
 \begin{pmatrix} \sH_1 \times \{0\} & R^{*} \\
 R  & \sH_2 \times \{0\} \end{pmatrix},
\end{equation}
is such a selfadjoint relation.
Here the diagonal entries stand for the zero
operators on $\sH_1$ and $\sH_2$, respectively. Likewise,
\begin{equation}\label{symmhh}
H=
\begin{pmatrix} \sH_1 \times \{0\} & \{0\} \times \{0\} \\
\sH_1 \times \sH_2  & \{0\} \times \sH_2 \end{pmatrix},
\end{equation}
is also a selfadjoint relation that extends $R$.
The entry $\{0\} \times \sH_2$ in this matrix
is a purely multivalued relation in $\sH_2$. That these block relations are actually
selfadjoint extensions of $R$ is based on the idea that the
block representation of $R$, when considered in the larger space Hilbert space
$\sH_1 \oplus \sH_2$, given by
\begin{equation}\label{symm0}
 S=
 \begin{pmatrix} \sH_1 \times \{0\} & \{0\} \times \{0\} \\
 R  & \{0\} \times \{0\} \end{pmatrix},
\end{equation}
defines a closed symmetric relation in $\sH_1 \oplus \sH_2$,
and that the block representation of its adjoint is then given by
\begin{equation}\label{symm0*}
S^{*}=
\begin{pmatrix} \sH_1 \times \{0\} & R^{*} \\
\sH_1 \times \sH_2  & \sH_2 \times \sH_2 \end{pmatrix}.
\end{equation}
The above observations are completely formal and need to be justified,
i.e., one needs to develop a calculus for $2 \times 2$ blocks of linear relations;
see Remark \ref{mulmat} and the text above it.

It is not difficult to see that the interpretation of the
symmetric relation $S$ in \eqref{symm0} leads to the following graph representation
\begin{equation}\label{symm}
S
=\left\{ \,\left\{  \begin{pmatrix} f_{1} \\ 0\end{pmatrix},
\begin{pmatrix} 0  \\ g_{2} \end{pmatrix} \right\} :
\, \{f_{1}, g_{2}\} \in R\, \right\}.
\end{equation}
It is clear that $S$ has the property $\dom S \perp \ran S$ and one can show that, in fact,
every relation with this property is of the form \eqref{symm}. The adjoint of $S$ is
given by
\begin{equation}\label{symm*}
 S^*
 =\left\{\, \left\{\begin{pmatrix} h_{1}  \\ h_{2} \end{pmatrix},
 \begin{pmatrix} k_{1}  \\ k_{2} \end{pmatrix} \right\} :
 \, h_{1} \in \sH_{1}, \, \{h_{2}, k_{1}\} \in R^*, \,k_{2} \in \sH_{2} \,  \right\};
\end{equation}
cf. \eqref{symm0*}.
By choosing an appropriate boundary triplet $\{\cG, \Gamma_0, \Gamma_1\}$
all selfadjoint extensions $A_\Theta$
of $S$ in $\sH$ can be parametrized by selfadjoint relations $\Theta$
in the parameter space $\cG$,
via
\[
A_\Theta=\ker ( \Gamma_1 -\Theta \Gamma_0).
\]
The selfadjoint extensions in  \eqref{symmkk} and \eqref{symmhh} correspond
to the parameter being the zero operator and the purely multivalued relation,
respectively.  In particular,  the Friedrichs extension $S_F$ and the
Kre\u{\i}n-von Neumann extension $S_K$  of $S$ will be determined.
In general they are not transversal with respect to $S$,
but they are transversal with respect to $S_F \cap S_K$.
This leads to a new boundary triplet by means of which the nonnegative
extensions are parametrized by nonnegative relations.
On the other hand, by introducing a symmetric extension of $S$ or, loosely speaking,
by making the parameter space smaller in an appropriated manner, it will be shown,
that depending on whether the operator part $R_{\rm s}$ of $R$ is bounded or not,
there is a correspondence between semibounded selfadjoint parameters $\Theta$ and
semibounded selfadjoint extensions $A_\Theta$, or not, respectively.

Here is an overview of the contents of the paper.
The notion of a linear block relation is introduced in Section \ref{tweee}.
This short treatment is all that is needed in this paper.
Section \ref{orth} contains a treatment of linear relations
whose domain and range are orthogonal.
In Section \ref{BT5} all selfadjoint extensions of $S$ are described
by means of an appropriate boundary triplet for $S^*$.
A brief intermezzo about nonnegative selfadjoint extensions is given in
Section \ref{mezzo}.
The Friedrichs and Kre\u{\i}n-von Neumann extensions
and related boundary triplets are studied in Section \ref{Fried}; see Proposition \ref{BTS0}.
A simple description of all nonnegative selfadjoint extensions of $S$ is given in Theorem \ref{BTS0op}
and there is a characterization of all extremal extensions of $S$ in Corollary \ref{extremal}.
The semibounded extensions of a certain symmetric extension of $S$
are studied in Section \ref{special}
by means of the asymptotic behaviour of an associated Weyl function.
This leads to the alternative mentioned above; see Theorem \ref{propsemibd}.

Blocks of linear relations are built on the treatment of columns and rows
of linear relations in \cite{HSSW}.
 For a related general treatment of blocks of linear operators, see \cite{MolSz};
see also \cite{Tr}. A characterization of linear relations as block relations
will be given later elsewhere; cf. \cite{La2003}.
Note that in the operator case the block in \eqref{symm} was mentioned
by Coddington in \cite{Co3} in connection with a paper of Hestenes \cite{He},
who considered selfadjoint operator extensions of arbitrary closed linear operators.
For more information in this case, see \cite{OS}.
The introduction of the corresponding symmetric relation in \eqref{symm},
with $R$ being a linear relation, goes back to \cite{Co3}.
The present paper may be seen as a special case of a general completion problem,
namely to complete the following block  of relations
\[
  \begin{pmatrix} * & * \\
 R  & *  \end{pmatrix},
\]
to a nonnegative selfadjoint relation in the Hilbert space $\sH=\sH_1 \oplus \sH_2$;
cf. \cite{Ha}.

\section {Linear relations with a block structure} \label{tweee}

Before formally introducing blocks of linear relations, here is a brief review
of the notions of column and row for pairs of linear relations; cf. \cite{HSSW}.
Let $\sH$, $\sK$, $\sH_i$, and $\sK_i$, $i=1,2$,  be Hilbert spaces.
Let $A$ be a linear relation from $\sH$ to $\sK_1$
and let $B$ be a linear relation from $\sH$ to $\sK_2$.
Then the \textit{column} $\col (A\,;\,B)$ of $A$ and $B$ as a relation from
$\sH$ to $\sK_1 \oplus \sK_2$ is defined by
\begin{equation}\label{col}
 \begin{pmatrix} A \\ B \end{pmatrix}
 = \left\{ \left\{ h, \begin{pmatrix} k_1 \\k_2 \end{pmatrix} \right\} :\,
 \{h,k_1\} \in A, \,\, \{h,k_2\} \in B \,\right\}.
\end{equation}
Observe that
\[
\begin{split}
&\dom \col(A\,;\,B)=\dom A \cap \dom B,\\
&\ker \col(A\,;\,B)=\ker A \cap \ker B, \\
&\ran \col(A\,;\,B)=\{ k_1 \oplus k_2:\, \{h,k_1\} \in A, \,\{h,k_2\} \in B\}, \\
& \mul \col(A\,;\,B)=\mul A \times \mul B.
\end{split}
\]
 The column of $A$ and $B$ resembles a sum of linear relations once
the range spaces of $A$ and $B$ are combined in the above way.
Moreover, if $A'$ is a linear relation from $\sH$ to $\sK_1$
and  $B'$ is a linear relation from $\sH$ to $\sK_2$,
such that $A \subset A'$ and $B \subset B'$, then by \eqref{col},
it is clear that the extensions are preserved in the sense of the column
\begin{equation}\label{colext}
 \begin{pmatrix} A \\ B \end{pmatrix} \subset \begin{pmatrix} A' \\ B' \end{pmatrix}.
\end{equation}
Next let $C$ be a linear relation from $\sH_1$ to $\sK$
and let $D$ be a linear relation from $\sH_2$ to $\sK$.
Then the \textit{row} $(C;D)$ of $C$ and $D$ as a relation from
$\sH_1 \oplus \sH_2$ to $\sH$ is defined by
\begin{equation}\label{row}
  (C\,; D)
 = \left\{ \left\{ \begin{pmatrix} h_1 \\h_2 \end{pmatrix}, k_1+k_2 \right\} :\,
 \{h_1,k_1\} \in C, \,\, \{h_2,k_2\} \in D \,\right\}.
\end{equation}
The row of $C$ and $D$ resembles a componentwise sum of linear relations
once the domain spaces of $C$ and $D$ are combined in the above way.
Observe that
\[
\begin{split}
&\dom (C \,;\, D)=\dom C \times \dom D,\\
&\ker (C\,;\,D)=\{ h_1 \oplus h_2:\, \{h,k_1\} \in C, \,\{h_2,-k_1\} \in D\}, \\
&\ran (C\,;\,D)=\ran C + \ran D, \\
&\mul (C\,;\,D)=\mul C + \mul D.
\end{split}
\]
 The following proposition goes back to \cite{HSSW},
where one can also find a simple proof. It may be helpful to mention
that the definition of an adjoint relation depends on the Hilbert spaces in which
the original relation is considered. Thus in each of the following statements
one should make sure what Hilbert spaces are involved.

\begin{proposition}\label{rowcol}
Let the relations $A$, $B$, $C$, and $D$ as above. Then the following statements hold.
\begin{enumerate} \def\labelenumi{\rm(\roman{enumi})}
\item The column of $A$ and $B$ satisfies
\[
  \begin{pmatrix} A \\ B \end{pmatrix}^* \supset  (A^* \,;\, B^*).
\]
\item The row of $C$ and $D$ satisfies
\[
   (C\,;\,D)^* =\begin{pmatrix} C^* \\ D^* \end{pmatrix}.
\]
\item
If $B$ is bounded and densely defined with $\dom A \subset \dom B$, there is
equality in {\rm (i)}.
\end{enumerate}
\end{proposition}

\begin{remark}
It follows directly from (iii) with $B=\sH \times \{0\}$, that
\[
 \begin{pmatrix} A \\  \sH \times \{0\} \end{pmatrix}^*
=(A^* \,;\, \sH \times \{0\}).
\]
There are more situations when equality prevails in (i).
For instance, if $\sM$ is a linear subspace in $\sK_2$,
and $B=\sH \times \sM$ one sees by
a direct argument that
\begin{equation}\label{rowcol+1}
 \begin{pmatrix} A \\  \sH \times \sM \end{pmatrix}^*
 =(A^*\,;\, \sM^\perp \times \{0\} ).
 \end{equation}
Recall that the domain of  $\col (A;B)$ is given
by $\dom A \cap \dom B$. Hence, if $\sM$ is a linear subspace
in $\sK_2$ and $B=\{0\} \times \sM$, then it follows that
\[
 \begin{pmatrix} A \\  \{0\} \times \sM \end{pmatrix}
 =\begin{pmatrix} \{0\} \times \mul A \\  \{0\} \times \sM \end{pmatrix}.
\]
A direct argument then shows that
\begin{equation*}
 \begin{pmatrix} A \\  \{0\} \times \sM \end{pmatrix}^*
=( \cdom A^* \times \sH \,;\,  \sM^\perp \times \sH )
\supset (  A^* \,;\, \sM^\perp \times \sH ),
\end{equation*}
with equality if and only if $\cdom A^* \times \sH =A^*$.
Thus, in general, there is no equality in (i).
For later use, observe that
\begin{equation}\label{rowcol+2}
 \begin{pmatrix} \{0\} \times \mul A \\  \{0\} \times \sM \end{pmatrix}^*
=( \cdom A^* \times \sH \,;\, \sM^\perp \times \sH ).
\end{equation}
\end{remark}

Now let the Hilbert space $\sH$ be decomposed into
two orthogonal components
$\sH_{1}$ and $\sH_{2}$
that are closed linear subspaces of $\sH=\sH_{1} \oplus \sH_{2}$.
Let
\[
 E_{ij} : \sH_{j} \to \sH_{i}, \quad i,j=1,2,
\]
be linear relations; they form a $2 \times 2$  \textit{block} of relations
$[E_{ij}]=[E_{ij}]_{i,j=1}^{2}$:
\begin{equation}\label{blok}
 [E_{ij}]
 =\begin{pmatrix} E_{11} & E_{12} \\
 E_{21} & E_{22} \end{pmatrix}.
\end{equation}
 Every block of relations  gives rise to a linear block relation in $\sH$.

\begin{definition}\label{matri}
Let $[E_{ij}]$ be a block as in \eqref{blok}.
Then the linear relation $E$ in $\sH$
\textit{generated} by the block
is defined as the row of its columns:
\begin{equation}\label{matrixnew}
 E=\begin{pmatrix}  \begin{pmatrix} E_{11} \\ E_{21} \end{pmatrix} ;
                                \begin{pmatrix} E_{12} \\ E_{22} \end{pmatrix}
     \end{pmatrix}.
\end{equation}
The relation $E$ is called the \textit{block relation} corresponding to the
block $[E_{ij}]$.
\end{definition}

Forming the row of the two columns in \eqref{matrixnew}
by means of \eqref{row} gives
\begin{equation}\label{matrix}
E=
\left\{
\left\{ \begin{pmatrix} f_{1} \\ f_{2}\end{pmatrix},
\begin{pmatrix} \alpha_{1} + \beta_{1} \\
\alpha_{2} + \beta_{2} \end{pmatrix} \right\} :\,
\begin{matrix}   \{f_{1}, \alpha_{1}\} \in E_{11}, \,\{f_{2}, \beta_{1}\} \in E_{12} \\
\{f_{1}, \alpha_{2}\} \in E_{21}, \,\{f_{2}, \beta_{2}\} \in E_{22} \end{matrix}
\right\},
\end{equation}
which is the natural way to think of the block relation $E$.
Observe that in the case where all of the relations $E_{ij}$
are everywhere defined bounded linear operators,
the block relation $E$ in \eqref{matrixnew} is the usual block operator.
It easily follows from the representation \eqref{matrix} of $E$ that
\[
 \dom \bE=( \dom E_{11} \cap \dom E_{21}) \oplus
 (\dom E_{12} \cap \dom E_{22} ),
\]
and that
\[
 \mul \bE= ( \mul E_{11} +\mul E_{12} ) \oplus
 (\mul E_{21} +\mul E_{22}  ).
\]
These two properties distinguish linear block relations among all relations in $\sH$.

In Definition \ref{matri} of a block relation one takes the row of two columns in
the block \eqref{blok}.
In the next lemma it is shown that one obtains the same block relation
when taking the column of the two rows in the block \eqref{blok}.

\begin{lemma}\label{matmat}
 Let $[E_{ij}]$ be a block as in \eqref{blok}. Then
\[
\begin{pmatrix}  ( E_{11} \,;\, E_{12} ) \\  (E_{21} \,;\, E_{22}) \end{pmatrix}
=
\begin{pmatrix}
 \begin{pmatrix} E_{11} \\ E_{21} \end{pmatrix}  ;
 \begin{pmatrix} E_{12} \\ E_{22} \end{pmatrix}
\end{pmatrix}.
\]
\end{lemma}

\begin{proof}
The definition of a column in \eqref{col} shows that
\[
 \begin{pmatrix}  ( E_{11} \,;\, E_{12} ) \\
  (E_{21} \,;\, E_{22}) \end{pmatrix}
 = \left\{
\left\{ f,
\begin{pmatrix} \gamma_{1}  \\ \gamma_{2} \end{pmatrix} \right\} :\,
\begin{matrix}   \{f, \gamma_{1}\} \in (E_{11}\,;\, E_{12})  \\
\{f, \gamma_{2}\} \in (E_{21} \,;\,E_{22} ) \end{matrix} \right\}.
\]
Recall that by the definition of a row in \eqref{row} one has
$\{f, \gamma_{1}\} \in (E_{11} \,;\, E_{12})$ if and only if
\[
\{f, \gamma_1\} =\left\{ \begin{pmatrix} f_1 \\ f_2 \end{pmatrix},  \alpha_1+ \beta_1 \right\}
\quad \mbox{with} \quad
\{f_1, \alpha_1\} \in E_{11} \quad \mbox{and} \quad \{f_2, \beta_1 \} \in E_{12},
\]
and, similarly, $\{f, \gamma_{2}\} \in (E_{21} \,;\, E_{22})$ if and only if
\[
\{f, \gamma_2\} =\left\{ \begin{pmatrix} f_1 \\ f_2 \end{pmatrix},  \alpha_2+ \beta_2 \right\}
\quad \mbox{with} \quad
\{f_1, \alpha_2\} \in E_{21} \quad \mbox{and} \quad \{f_2, \beta_2 \} \in E_{22}.
\]
Combining these facts, one sees that $\{f, \gamma_{1}\} \in (E_{11}\,;\,E_{12})$
and $\{f, \gamma_{2}\} \in (E_{21}\,;\, E_{22})$ if and only if
\[
\left\{ f,  \begin{pmatrix} \gamma_{1}  \\ \gamma_{2} \end{pmatrix} \right\}
=\left\{ \begin{pmatrix} f_1 \\ f_2 \end{pmatrix},
\begin{pmatrix} \alpha_1+ \beta_1 \\ \alpha_2 +\beta_2 \end{pmatrix} \right\}
\quad \mbox{with} \quad
\begin{matrix} \{f_1, \alpha_1\} \in E_{11}, \, \{f_2, \beta_1 \} \in E_{12}, \\
\{f_1, \alpha_2\} \in E_{21}, \, \{f_2, \beta_2 \} \in E_{22}. \end{matrix}
\]
This shows the identity thanks to \eqref{matrix}.
\end{proof}

Let $[E_{ij}],  [F_{ij}]$
be blocks of the form \eqref{blok}
and let $E$ and $F$ be the linear block relations in $\sH$ generated by them.
The blocks are said to satisfy the \textit{inclusion}
$[E_{ij}] \subset [F_{ij}] $ if $E_{ij} \subset F_{ij}$ for all $i,j$.
It follows from \eqref{matrix} that
\[
[E_{ij}] \subset [F_{ij}]  \quad \Rightarrow \quad E \subset F.
\]
Likewise,
let $[E_{ij}]$  be a block of the form \eqref{blok}.
Then the $2 \times 2$ block $[E_{ij}]^*$
of the adjoint relations (formal adjoint) is defined by
\[
 [E_{ij}]^*=\begin{pmatrix} E_{11}^* & E_{21}^* \\
 E_{12}^* & E_{22}^* \end{pmatrix},
\]
where $E_{ij}^*$ is a closed linear relation from  $\sH_{i}$
to $\sH_{j}$, $i,j=1,2$.
Thus one sees that also $[E_{ij}]^*$  is a block of the form \eqref{blok}.
In general, there is the following inclusion result.

\begin{proposition}
 Let $[E_{ij}]$ be a block as in \eqref{blok}. Then
\begin{equation}\label{adjstar+}
\begin{pmatrix}
  \begin{pmatrix} E_{11}^* \\ E_{12}^* \end{pmatrix} ;
  \begin{pmatrix} E_{21}^* \\ E_{22}^* \end{pmatrix}
\end{pmatrix}
\subset
\begin{pmatrix}
\begin{pmatrix} E_{11} \\ E_{21} \end{pmatrix}  ;
\begin{pmatrix} E_{12} \\ E_{22} \end{pmatrix}
\end{pmatrix}^*.
\end{equation}
\end{proposition}

\begin{proof}
It follows from (ii) of Proposition \ref{rowcol} that
\[
\begin{pmatrix}
   \begin{pmatrix} E_{11} \\ E_{21} \end{pmatrix};
   \begin{pmatrix} E_{12} \\ E_{22} \end{pmatrix}
\end{pmatrix}^*
=
\begin{pmatrix}
   \begin{pmatrix} E_{11} \\ E_{21} \end{pmatrix}^* \vspace{0.2cm} \\
   \begin{pmatrix} E_{12} \\ E_{22} \end{pmatrix}^*
\end{pmatrix}.
\]
Likewise, the following inclusions are obtained
from (i) of Proposition \ref{rowcol}:
\[
\begin{pmatrix} E_{11} \\ E_{21} \end{pmatrix}^*
\supset \left( E_{11}^* \,;\, E_{21}^* \right)
\quad \mbox{and} \quad
\begin{pmatrix} E_{12} \\ E_{22} \end{pmatrix}^*
\supset \left( E_{12}^* \,;\, E_{22}^* \right).
\]
These two inclusions may be combined by \eqref{colext}, which gives
\[
\begin{pmatrix}
   \begin{pmatrix} E_{11} \\ E_{21} \end{pmatrix}^* \vspace{0.2cm} \\
   \begin{pmatrix} E_{11} \\ E_{21} \end{pmatrix}^*
\end{pmatrix}
\supset
\begin{pmatrix}
\left( E_{11}^* \,;\, E_{21}^* \right)    \\ \left( E_{12}^* \,;\, E_{22}^* \right)
\end{pmatrix}.
\]
By Lemma \ref{matmat}, one sees that
\[
\begin{pmatrix}
\left( E_{11}^* \,;\, E_{21}^* \right)    \\ \left( E_{12}^* \,;\, E_{22}^* \right)
\end{pmatrix}=
\begin{pmatrix}
  \begin{pmatrix} E_{11}^* \\ E_{12}^* \end{pmatrix} ;
  \begin{pmatrix} E_{21}^* \\ E_{22}^* \end{pmatrix}
\end{pmatrix},
\]
which completes the proof.
\end{proof}

As to equality in \eqref{adjstar+}, there are the following sufficient conditions;
cf. Proposition \ref{rowcol}
and   the identies in \eqref{rowcol+1} and \eqref{rowcol+2}.

\begin{corollary}
 Let $[E_{ij}]$ be a block as in \eqref{blok}.
Assume that, up to interchange of $A$ and $B$,
the entries of each column $\col (A;B)$ in $[E_{ij}]$ satisfy one of the following:
\begin{enumerate} \def\labelenumi{\rm(\roman{enumi})}
\item  the condition {\rm (iii)}  in Proposition \ref{rowcol};
\item $B=\sH \times \sK_2$;
\item $A$ is purely singular and $B=\{0\} \times \sM$.
\end{enumerate}
Then there is equality in \eqref{adjstar+}.
\end{corollary}

The following observation concerns a useful property of a class
of singular relations in $\sH=\sH_1 \oplus \sH_2$.

\begin{corollary}\label{write}
Let $\sM_1, \sN_1 \subset \sH_1$ and $\sM_2, \sN_2 \subset \sH_2$
be closed linear subspaces. Then
\begin{equation}\label{write1}
  \begin{pmatrix} \sM_1 \times \sN_1 & \sM_2 \times  \sN_1 \\
 \sM_1 \times \sN_2 & \sM_2 \times \sN_2 \end{pmatrix}
 =
 (\sM_1 \oplus \sM_2)^\top \times (\sN_1 \oplus \sN_2)^\top.
\end{equation}
Moreover, if $\sN_1=\sM_1^\perp$ and $\sN_2=\sM_2^\perp$, then
the relation \eqref{write1} is selfadjoint.
\end{corollary}

\begin{proof}
The identity \eqref{write1} follows directly from Definition \ref{matri}; see \eqref{matrix}.
The second statement is clear from \eqref{write1}, since one sees by
a direct argument that for any closed subspace $\sL$ of a Hilbert space $\sH$ 
the linear relation $\sL\oplus\sL^\perp$ is selfadjoint in $\sH$.  
\end{proof}

Here the notation $(\sM_1 \oplus \sM_2)^\top$ is a shortcut for the vector notation
\[
 \begin{matrix} \sM_1 \\ \oplus  \\ \sM_2 \end{matrix}
 =\left\{ \begin{pmatrix} h_1 \\ h_2 \end{pmatrix} :
                       h_1 \in \sM_1,\, h_2 \in \sM_2 \right\}
 =(\sM_1 \oplus \sM_2)^\top.
 \]
Hence, $ (\sM_1 \oplus \sM_2)^\top \times (\sN_1 \oplus \sN_2)^\top$ means
\[
(\sM_1 \oplus \sM_2)^\top \times (\sN_1 \oplus \sN_2)^\top=
\left\{ \left\{ \begin{pmatrix} h_1 \\h_2 \end{pmatrix},
\begin{pmatrix} k_1 \\k_2 \end{pmatrix} \right\}
 :\, h_i \in \sM_i, \, k_i \in \sN_i, \,i=1,2 \, \right\}.
\]

As a consequence of the above observations, one sees
that the block relations  \eqref{symm0} and \eqref{symm0*}
are well-defined, and that  \eqref{symm0*} is the adjoint of
\eqref{symm0}, so that \eqref{symm0} is symmetric.
It follows from Definition \ref{matri} that the relations defined by
\eqref{symm0} and in \eqref{symm} coincide. A similar statement holds for
the equality of \eqref{symm0*} and \eqref{symm*}.
Furthermore, one sees that the block relations \eqref{symmkk}
and \eqref{symmhh} are well-defined and selfadjoint.

\begin{remark}\label{mulmat}
It should be observed that the block representation of a linear relation
need not be unique. Note, as an example,  that $K$ in \eqref{symmkk} is equal
to the block relation
\begin{equation}\label{symmkkk}
  \begin{pmatrix} \cdom R \times \mul R^* & R^{*} \\
 R  & \cdom R^* \times \mul R \end{pmatrix},
\end{equation}
since \eqref{symmkk} and \eqref{symmkkk} are well-defined selfadjoint block relations,
and  \eqref{symmkk} is included in \eqref{symmkkk}.
To appreciate this equality, consider, for instance, the left upper corner
$\cdom R \times \mul R^*$ in \eqref{symmkkk},
which is a selfadjoint singular relation.
The elements in $\{0\} \times \mul R^*$ already appear
in the right upper corner, whereas $\cdom R \times \{0\}$
has a domain which includes the domain of the left bottom corner.
Hence, replacing $\cdom R \times \mul R^*$ by the selfadjoint
relation $\sH_1 \times \{0\}$ gives the same block relation.
 \end{remark}

\section{Linear relations whose domain and range are orthogonal}\label{orth}

Let $S$ be a linear relation in a Hilbert space $\sH$.
The interest will be in the rather special case that $\dom S \perp \ran S$.
Clearly, if $S$ has this property, then the same is true
for the inverse relation $S^{-1}$.
Note that the orthogonality  condition is always satisfied when either
$\dom S =\{0\}$ or $\ran S=\{0\}$. Here the orthogonality property
will be characterized in two different ways. \\

Recall that the \textit{numerical range} $\cW(S)$ of a linear
relation $S$ in $\sH$ is defined by
\[
\cW(S) = \{(g,f):\, \{f,g\}\in S:\, \|f\| = 1\,\} \subset \dC
\]
when $\dom S \neq \{0\}$, and by $\{0\}\subset\dC$
if $\dom S = \{0\}$, i.e. if $S$ is purely multivalued.
 It is clear that all eigenvalues in $\dC$ of $S$ belong
 to its numerical range $\cW(S)$.
 Moreover, for linear relations the numerical range is a convex set;
see \cite[Proposition~2.18]{HdSSz09}.
Clearly, the numerical range of the inverse of $S$ is given by
\[
\cW(S^{-1}) = \{ \lambda\in \dC: \overline\lambda \in \cW(S)\,\}.
\]
Here is the first characterization.

\begin{lemma}\label{numranzero}
Let $S$ be a linear relation  in $\sH$.
Then the following statements are equivalent:
\begin{enumerate} \def\labelenumi{\rm(\roman{enumi})}
\item $\dom S\perp \ran S$;
\item $\cW(S)=\{0\}$.
\end{enumerate}
\end{lemma}

\begin{proof}
(i) $ \Rightarrow$ (ii) This implication is clear from the definition of $\cW(S)$.

(ii) $ \Rightarrow$ (i) To prove this reverse implication
the following modification of polarization identity is needed:
for all $\{f_1,g_1\}, \{f_2,g_2\}\in S$ one has
\begin{equation}\label{polar}
\begin{split}
   (g_1,f_2)& = \dfrac{1}{4}\, \big[(g_1+g_2,f_1+f_2)-(g_1-g_2,f_1-f_2)  \\
   &\hspace{1.5cm}  +i(g_1+ig_2,f_1+if_2)- i(g_1-ig_2,f_1-if_2) \big].
\end{split}
\end{equation}
Now assume that $f_1\in \dom S$ and $g_2\in \ran S$.
Then $\{f_1,g_1\}, \{f_2,g_2\}\in S$ for some $g_1,f_2\in\sH$.
Hence if (ii) holds, then the left-hand side of \eqref{polar}
shows that $(g_1,f_2)=0$ and thus $\dom S\perp \ran S$.
\end{proof}

Thus, if $\dom S \perp \ran S$, then  it is clear that the relation $S$
is symmetric and that only $\lambda=0$ can be an eigenvalue of $S$.
In fact, the orthogonality property implies that $S$ is
semibounded; for instance, $S$ is semibounded from below
with lower bound $m(S)=0$. \\

The following result is a characterization of the
linear relation in \eqref{symm0} and \eqref{symm}:
it shows that one can express the results
in terms of $R$ or $S$.

\begin{lemma}
Let $S$ be a linear relation  in $\sH$.
Then the following statements are equivalent:
\begin{enumerate} \def\labelenumi{\rm(\roman{enumi})}
\item $\dom S\perp \ran S$;
\item  $\sH=\sH_1 \oplus \sH_2$ and there exists a linear relation $R$
from  $\sH_1$ to  $\sH_2$, such that
\begin{equation}\label{tstar0}
S
=\left\{ \,\left\{  \begin{pmatrix} f_{1} \\ 0\end{pmatrix},
\begin{pmatrix} 0  \\ g_{2} \end{pmatrix} \right\} :
\, \{f_{1}, g_{2}\} \in R\, \right\}.
\end{equation}
\end{enumerate}
\end{lemma}

\begin{proof}
(i) $\Rightarrow$ (ii)
Assume that  $\dom S \perp \ran S$. Then
choose an orthogonal decomposition $\sH=\sH_1 \oplus \sH_2$,
such that  $\dom S \subset \sH_1$ and $\ran S \subset \sH_2$.
Define the linear relation $R$ from $\sH_1$ to $\sH_2$ by
 \[
R=\left\{ \{f_1,g_2\}\in\sH_1\times \sH_2:\,
\left\{ \begin{pmatrix} f_1\\ 0 \end{pmatrix},
\begin{pmatrix} 0 \\ g_2 \end{pmatrix}  \right\} \in S\right\}.
\]
It follows that $S$ is of the form \eqref{symm}.
Of course, the choice $\dom S \subset \sH_1$
and $\ran S \in \sH_2$ is arbitrary:
one may also interchange the spaces which results
in taking the inverse of $S$.

(ii) $\Rightarrow$ (i) This implication is clear.
\end{proof}

Note that the relation $S$ in $\sH$ defined in \eqref{tstar0}
is closed if and only if
the relation $R$ from $\sH_1$ to $\sH_2$ is closed.  \\

In the rest of the paper the attention is restricted
to linear relations in $\sH$ for which $\dom S \perp \ran S$ or,
equivalently, $\cW(S)=\{0\}$.
In this case $S$ is of the form \eqref{tstar0}.
The elements of $R$ as a linear relation from $\sH_1$ to $\sH_2$
will be denoted by $\{f_1,f_2\}$,
but frequently, depending on the situation,
also in vector notation by
\[
\begin{pmatrix} f_1 \\ f_2 \end{pmatrix},
\quad \mbox{where} \quad f_1 \in \sH_1, \quad f_2 \in \sH_2.
\]
The adjoint  $R^{*}$ is a closed linear relation from $\sH_2$ to $\sH_1$.
Hence, if $R$ is closed, then it is clear that
\begin{equation}\label{decomp}
\sH_1 \oplus \sH_2 =R \hoplus R^{\perp},
\end{equation}
which is an orthogonal decomposition of $\sH_1 \oplus \sH_2$, where
\begin{equation}\label{decompl}
R^{\perp}=JR^{*}
=\left\{ \,\begin{pmatrix} \beta \\ -\alpha\end{pmatrix}: \,
\{\alpha, \beta\} \in R^{*} \,\right\},
\end{equation}
and $J$ stands for the flip-flop operator $J\{\varphi, \psi\}=\{\psi, -\varphi\}$.

\section{A boundary triplet generated by a closed linear relation}\label{BT5}

Let $S$ be a closed linear relation in a Hilbert space $\sH$ for which
$\dom S \perp \ran S$. Then $\sH=\sH_1 \oplus \sH_2$ and there
exists a closed linear relation $R$ from $\sH_1$ to $\sH_2$ such that
$S$ is given by \eqref{tstar0}. In order to describe the selfadjoint
extensions of $S$ in $\sH$ a suitable boundary triplet will be chosen for $S^*$.
A first step is the determination of the adjoint $S^*$ of $S$  below.

\begin{lemma}
Let $R$ be a closed linear relation from $\sH_1$ to $\sH_2$ and let $S$
be the closed symmetric relation defined in  \eqref{tstar0}. Then
\begin{equation}\label{tstar}
 S^*
 =\left\{\, \left\{\begin{pmatrix} h_{1}  \\ h_{2} \end{pmatrix},
 \begin{pmatrix} k_{1}  \\ k_{2} \end{pmatrix} \right\} :
 \, h_{1} \in \sH_{1}, \, \{h_{2}, k_{1}\} \in R^*, \,k_{2} \in \sH_{2} \,  \right\}.
\end{equation}
\end{lemma}

\begin{proof}
The assertion follows immediately from the identity
\[
\left( \begin{pmatrix} k_{1} \\ k_{2} \end{pmatrix} ,
\begin{pmatrix} f_{1}  \\ 0 \end{pmatrix} \right)
- \left( \begin{pmatrix} h_{1}  \\ h_{2} \end{pmatrix},
\begin{pmatrix} 0  \\ g_{2} \end{pmatrix} \right)
=(k_1, f_1)-(h_2,g_2).
\]
This identity shows that the right-hand side of\eqref{tstar}
is contained in the adjoint $S^*$,
as $(k_1, f_1)-(h_2,g_2)=0$ for all $\{f_1,g_2 \} \in R$
and $\{h_2,k_1\} \in R^*$.
The adjoint relation $S^*$ is contained
in the right-hand side of \eqref{tstar}
 as $(k_1, f_1)=(h_2,g_2)$ for all $\{f_1,g_2\} \in R$
implies that $\{h_2,k_1\} \in R^*$.
\end{proof}

For $\lambda \in \dC$ the eigenspace associated
with \eqref{tstar} is given by
\[
 \wh \sN_{\lambda}(S^*)
 = \left\{ \,\left\{\begin{pmatrix} h_{1}  \\ h_{2} \end{pmatrix},
 \begin{pmatrix} k_{1}  \\ k_{2} \end{pmatrix} \right\}:\, k_{1}
 =\lambda h_{1}, \, k_{2}=\lambda h_{2}, \,  \,\{h_{2}, k_{1}\} \in R^* \,\right\},
\]
and,  hence, with $\sN_\lambda(S^*)=\ker (S^*-\lambda)$, one has
\[
 \sN_{\lambda}(S^*)= \left\{ \begin{pmatrix} h_{1}  \\ h_{2} \end{pmatrix} :
 \,\{h_{2},\lambda h_{1}\} \in R^* \,\right\}.
\]
Likewise, the multivalued part of $S^*$ is given by
\[
 \mul S^*=  \left\{ \begin{pmatrix} k_{1}  \\ k_{2} \end{pmatrix} : \,
 k_1 \in \mul R^*, \, k_2 \in \sH_2 \,\right\}.
 \]

The particular form of $S^{*}$ in \eqref{tstar} leads to
 a ``natural'' boundary triplet for $S^{*}$; cf.  \cite{BHS}, \cite{DM}.
 For this, one needs to define a parameter space $\cG$,
and it turns out that
 \begin{equation}\label{hk00}
 \cG=   R^\perp = \left\{ \begin{pmatrix} h_{1}  \\ h_{2} \end{pmatrix} :
 \,\{h_{2},-h_{1}\} \in R^* \,\right\}= \sN_{-1}(S^{*}),
\end{equation}
is an appropriate candidate, where
$R^\perp=(\sH_1 \oplus \sH_2) \ominus R$.
It is useful to observe that for $\{h_1, h_2 \} \in \cG$
there are the following trivial equivalences:
\[
h_2=0 \quad \Leftrightarrow \quad h_1 \in \mul R^*,
\]
and, likewise
\[
h_1=0 \quad \Leftrightarrow \quad h_2 \in \ker R^*.
\]
Let $Q$ be the orthogonal projection
from $\sH_1 \oplus \sH_2$ onto $\cG$.

\begin{theorem}\label{BT}
Let $R$ be a closed linear relation from $\sH_1$ to
$\sH_2$ and let $S$
be the symmetric relation defined in  \eqref{tstar0}
with adjoint \eqref{tstar}.
Let $Q$ be the orthogonal projection from
$\sH_1 \oplus \sH_2$ onto $\cG$
in \eqref{hk00}. Assume that
\begin{equation}\label{hk}
 \left\{\begin{pmatrix} h_{1}  \\ h_{2} \end{pmatrix},
 \begin{pmatrix} k_{1}  \\ k_{2} \end{pmatrix} \right\},
 \quad \{h_{2}, k_{1}\} \in R^{*},
\end{equation}
is an element in $S^{*}$ and define
\begin{equation}\label{bt}
 \Gamma_0  \left\{\begin{pmatrix} h_{1}  \\ h_{2} \end{pmatrix},
 \begin{pmatrix} k_{1}  \\ k_{2} \end{pmatrix} \right\}
 =\begin{pmatrix} -k_{1}  \\ h_{2} \end{pmatrix}
\quad \mbox{and} \quad
\Gamma_1  \left\{\begin{pmatrix} h_{1}  \\ h_{2} \end{pmatrix},
\begin{pmatrix} k_{1}  \\ k_{2} \end{pmatrix} \right\}
=Q  \begin{pmatrix} h_{1} \\ k_{2} \end{pmatrix}.
\end{equation}
Then $\Gamma_{0}$ and $\Gamma_{1}$ are mappings
from $S^{*}$ onto $\cG$
and $\{ \cG, \Gamma_0, \Gamma_1\}$
is a boundary triplet for the relation $S^{*}$.
\end{theorem}

\begin{proof}
Observe for the element in \eqref{hk} that
$\{h_{2}, k_{1}\} \in R^{*}$ by definition, so that
by \eqref{hk00} one concludes that
\[
 \begin{pmatrix} -k_{1} \\ h_{2} \end{pmatrix} \in \cG.
\]
Note that $\Gamma_0$ and $\Gamma_1$ map $S^*$ into $\cG$.
Therefore, for general elements in $S^{*}$ of the form
\[
 \left\{\begin{pmatrix} h_{1}  \\ h_{2} \end{pmatrix},
 \begin{pmatrix} k_{1}  \\ k_{2} \end{pmatrix} \right\},
 \quad
 \left\{\begin{pmatrix} f_{1}  \\ f_{2} \end{pmatrix},
 \begin{pmatrix} g_{1}  \\ g_{2} \end{pmatrix} \right\},
\]
one has the Green identity
\[
\begin{split}
&\left( \begin{pmatrix} k_{1} \\ k_{2} \end{pmatrix} ,
\begin{pmatrix} f_{1}  \\ f_{2} \end{pmatrix} \right)
- \left( \begin{pmatrix} h_{1}  \\ h_{2} \end{pmatrix},
\begin{pmatrix} g_{1}  \\ g_{2} \end{pmatrix} \right)  \\
  &\hspace{1cm} =\left( \begin{pmatrix} h_{1} \\ k_{2} \end{pmatrix},
 \begin{pmatrix} -g_{1} \\ f_{2} \end{pmatrix} \right)
 -\left( \begin{pmatrix} -k_{1} \\h_{2} \end{pmatrix},
 \begin{pmatrix} f_{1} \\ g_{2} \end{pmatrix} \right)\\
 &\hspace{1cm} =\left( \begin{pmatrix} h_{1} \\k_{2} \end{pmatrix},
 Q \begin{pmatrix} -g_{1} \\f_{2} \end{pmatrix} \right)
 -\left( Q \begin{pmatrix} -k_{1} \\h_{2} \end{pmatrix},
 \begin{pmatrix} f_{1} \\ g_{2} \end{pmatrix} \right)\\
 &\hspace{1cm} =\left( Q\begin{pmatrix} h_{1} \\k_{2} \end{pmatrix},
  \begin{pmatrix} -g_{1} \\f_{2} \end{pmatrix} \right)
 -\left(  \begin{pmatrix} -k_{1} \\h_{2} \end{pmatrix},
 Q \begin{pmatrix} f_{1} \\ g_{2} \end{pmatrix} \right).
\end{split}
\]
Thus the abstract Green identity holds with the mappings
$\Gamma_{0}$ and $\Gamma_{1}$ in \eqref{bt}.

It is clear from the definition of $S^{*}$ that the mapping
$\Gamma_{0}$ is onto $\cG$.
Furthermore, in the definition of $S^{*}$ the elements
$h_{1} \in \sH_1$ and $k_{2} \in \sH_2$ are arbitrary;
in particular one can choose them as an arbitrary pair
in $\cG=\sN_{-1}(S^{*})$. Hence, the joint mapping
$(\Gamma_{0},\Gamma_{1})$ takes
$S^{*}$ onto $\cG\times \cG$. Consequently,
$\{ \cG, \Gamma_0, \Gamma_1\}$ is a boundary triplet
for the relation $S^*$.
 \end{proof}

The boundary triplet in \eqref{bt} determines
a pair of selfadjoint extensions of $S$.
In particular, $H=\ker \Gamma_{0}$
is a selfadjoint extension of $S$ given by
\begin{equation}\label{hh}
 H=\left\{\, \left\{  \begin{pmatrix} h_{1} \\ 0\end{pmatrix},
 \begin{pmatrix} 0  \\ k_{2} \end{pmatrix} \right\} :
 \, h_{1} \in \sH_1, \, k_{2} \in \sH_2\, \right\},
\end{equation}
and $m(H)=0$.   It is clear that
$H$ is a singular relation as
\[
 H= (\sH_1 \oplus \{0\} )^\top \times (\{0\} \oplus \sH_2)^\top;
\]
cf. \cite{HSS2018}. Note that $H$ coincides
with the block relation \eqref{symmhh}.
Clearly,  the spectrum of $H$
consists only of the eigenvalue $0 \in \sigma_{\rm p}(H)$,
so that $\rho(H)=\dC \setminus \{0\}$. Note that  for $\lambda \neq 0$,
it follows from the identity
\[
 \left\{\begin{pmatrix} h_{1}  \\ h_{2} \end{pmatrix},
 \begin{pmatrix} k_{1}  \\ k_{2} \end{pmatrix} \right\}
 =\left\{\begin{pmatrix} h_{1}-\frac{1}{\lambda} k_1  \\ 0 \end{pmatrix},
 \begin{pmatrix} 0  \\ k_{2}-\lambda h_2 \end{pmatrix} \right\}
 +\left\{\begin{pmatrix} \frac{1}{\lambda}k_{1}  \\ h_{2} \end{pmatrix},
 \begin{pmatrix} k_{1}  \\ \lambda h_{2} \end{pmatrix} \right\},
\]
together with \eqref{tstar},  \eqref{hh}, and \eqref{hk00}, that
 \[
 S^*=H \hplus \wh \sN_{\lambda}(S^*), \quad \lambda \neq 0.
\]
 It is straightforward to see that for $\varphi_{1} \in \sH_1$
and $\varphi_{2} \in \sH_2$ one has
\[
 (H-\lambda)^{-1} \begin{pmatrix} \varphi_{1} \\ \varphi_{2} \end{pmatrix}
 =\begin{pmatrix} -\frac{1}{\lambda} \varphi_{1} \\ \varphi_{2} \end{pmatrix},
 \quad \lambda \in \dC\setminus \{0\}.
\]
These preparations lead to the descriptions for the $\gamma$-field
and the Weyl function corresponding to the boundary triplet in \eqref{bt}.

\begin{theorem}\label{WEYL}
Let $R$ be a closed linear relation from $\sH_1$ to $\sH_2$
and let $S$
be the symmetric relation defined in  \eqref{tstar0}.
Let $Q$ be the orthogonal projection from
$\sH_1 \oplus \sH_2$ onto $\cG$
in \eqref{hk00}.
Let the boundary triplet $\{ \cG, \Gamma_{0}, \Gamma_{1}\}$
be given by \eqref{bt}. Then the corresponding $\gamma$-field
and Weyl function are given by
 \begin{equation}\label{gamm}
 \gamma(\lambda)
 =\begin{pmatrix} -\frac{1}{\lambda} & 0\\ 0 & 1 \end{pmatrix} \uphar \cG,
 \quad
 M(\lambda)
 =Q \begin{pmatrix} -\frac{1}{\lambda} & 0\\ 0 & \lambda \end{pmatrix}  \uphar \cG,
 \quad \lambda \in \dC \setminus \{0\}.
\end{equation}
\end{theorem}

\begin{proof}
Recall that for any $\lambda \in \dC$ one has that
\[
 \wh \sN_{\lambda}(S^*)
 = \left\{ \,\left\{\begin{pmatrix} h_{1}  \\ h_{2} \end{pmatrix},
 \begin{pmatrix} k_{1}  \\ k_{2} \end{pmatrix} \right\}:
 \,\{h_{2}, k_{1}\} \in R^*, \,\, \,k_{1}=\lambda h_{1}, \,\, k_{2}=\lambda h_{2}\right\}.
\]
Hence, for the elements in $\wh \sN_{\lambda}(S^*)$
it follows from \eqref{bt} that
\[
 \Gamma_{0} \left\{\begin{pmatrix} h_{1}  \\ h_{2} \end{pmatrix},
 \begin{pmatrix} k_{1}  \\ k_{2} \end{pmatrix} \right\}
 =\begin{pmatrix} -\lambda h_{1}  \\ h_{2} \end{pmatrix},
  \quad
  \Gamma_{1} \left\{\begin{pmatrix} h_{1}  \\ h_{2} \end{pmatrix},
 \begin{pmatrix} k_{1}  \\ k_{2} \end{pmatrix} \right\}
 =Q  \begin{pmatrix} h_{1} \\ \lambda h_{2} \end{pmatrix}.
\]
Therefore, by definition,  the graph of the Weyl function $M$ is given by
\[
 M(\lambda)=\left\{ \left\{ \begin{pmatrix} -\lambda h_{1} \\ h_{2} \end{pmatrix},
 Q \begin{pmatrix} h_{1} \\ \lambda h_{2} \end{pmatrix} \right\} :\,
 \{h_{2}, \lambda h_{1} \} \in R^{*}\,\right\},
\]
or, equivalently, replacing $-\lambda h_{1}$ by $h_{1}$,
\[
 M(\lambda)=\left\{ \left\{ \begin{pmatrix} h_{1} \\ h_{2} \end{pmatrix},
 Q \begin{pmatrix} -\frac{1}{\lambda} h_{1} \\ \lambda h_{2} \end{pmatrix} \right\} :\,
 \{h_{2}, - h_{1} \} \in R^{*}\,\right\}.
\]
Likewise, by definition, the graph of the $\gamma$-field is given by
\[
 \gamma(\lambda)
 =\left\{ \left\{ \begin{pmatrix} -\lambda h_{1} \\ h_{2} \end{pmatrix},
 \begin{pmatrix} h_{1} \\ h_{2} \end{pmatrix} \right\} :\,
 \{h_{2}, \lambda h_{1} \} \in R^{*}\,\right\},
\]
or, equivalently, replacing $-\lambda h_{1}$ by $h_{1}$,
\[
 \gamma(\lambda)=\left\{ \left\{ \begin{pmatrix} h_{1} \\ h_{2} \end{pmatrix},
 \begin{pmatrix} -\frac{1}{\lambda} h_{1} \\ h_{2} \end{pmatrix} \right\} :\,
 \{h_{2}, - h_{1} \} \in R^{*}\,\right\}.
\]
This completes the proof.
\end{proof}

The structure of the Weyl function $M$ in \eqref{gamm} gives the following
result immediately.

\begin{corollary}\label{cor4.4}
The Weyl function $M$ satisfies the weak identity
\[
 \left( M(\lambda)  \begin{pmatrix} h_{1} \\ h_{2} \end{pmatrix},
  \begin{pmatrix} h_{1} \\ h_{2} \end{pmatrix} \right)
=  -\frac{1}{\lambda} \,(h_{1}, h_{1})+ \lambda \,(h_{2}, h_{2}),
\quad   \begin{pmatrix} h_{1} \\  h_{2}  \end{pmatrix} \in \cG,
\]
where $\lambda \in \dC \setminus \{0\}$.
\end{corollary}

In particular, the identity holds for $\lambda < 0$,
so that $\lambda \mapsto M(\lambda)$
 is a nondecreasing function on $(-\infty,0)$.
 The limits
$M(-\infty)$ and $M(0)$ exist in the strong resolvent sense.
Their particular form can be found via
the asymptotic behavior of $M$ near $\lambda=-\infty$ and near $\lambda=0$.\\

The boundary triplet in Theorem \ref{BT} can be used to parametrize all
selfadjoint extensions of $S$ in \eqref{tstar0}. In fact,
the selfadjoint extensions $A$ of $S$ are in one-to-one correspondence with
the selfadjoint relations $\Theta$ in $\cG$, via
\begin{equation}\label{btbtnew}
A_\Theta=\ker ( \Gamma_1 -\Theta \Gamma_0),
\end{equation}
i.e., in other words
\begin{equation}\label{btbt}
A_\Theta
 =\left\{\, \left\{\begin{pmatrix} h_{1}  \\ h_{2} \end{pmatrix},
 \begin{pmatrix} k_{1}  \\ k_{2} \end{pmatrix} \right\} :\,
  \{h_{2}, k_{1}\} \in R^*,
\,
\left\{ \begin{pmatrix} -k_{1}  \\ h_{2} \end{pmatrix},
Q  \begin{pmatrix} h_{1} \\ k_{2} \end{pmatrix} \right\} \in \Theta
 \right\}.
\end{equation}
In particular, the relation $\Theta=\{0\} \times \cG $   is selfadjoint in $\cG$
and corresponds to the selfadjoint extension $H=\ker \Gamma_{0}$ in \eqref{hh}.
Likewise, the relation $\Theta =\cG \times \{0\}$, i.e., $\Theta =0$, is
selfadjoint in $\cG$ and
corresponds to the selfadjoint extension given by
\begin{equation}\label{kk}
 K=\left\{ \,\left\{\begin{pmatrix} h_{1}  \\ h_{2} \end{pmatrix},
 \begin{pmatrix} k_{1}  \\ k_{2} \end{pmatrix} \right\} :\,
 \{h_{1}, k_{2} \} \in R,\,\,  \{h_{2}, k_{1}\} \in R^*\, \right\},
\end{equation}
whose block representation is given by \eqref{symmkk};
cf. \eqref{symmkkk}. In general, the relation $K$ is not semibounded,
since $(k_2,h_2)=(h_1,k_1)$ implies
\[
 \left( \begin{pmatrix} k_{1}  \\ k_{2} \end{pmatrix},
 \begin{pmatrix} h_{1}  \\ h_{2} \end{pmatrix} \right)=(k_1,h_1)+(k_2,h_2)=2 \RE (k_1,h_1),
\]
which, in general, has no fixed sign.
It is clear from \eqref{tstar0}, \eqref{tstar},  \eqref{hh}, and \eqref{kk},
 that the selfadjoint extensions $H$ and $K$ are transversal, i.e.,
\[
S^{*}=H \hplus K,
\]
which, of course, agrees with the identities $H=\ker \Gamma_{0}$
and $K=\ker \Gamma_{1}$;
cf. \cite{DM}, \cite{BHS}.

\section{On nonnegative selfadjoint extensions of nonnegative relations}\label{mezzo}

Let $S$ be nonnegative relation in a Hilbert space $\sH$, in other words,
$(g,f) \geq 0$ for all $\{f,g\} \in S$.
Such a relation $S$ determines a nonnegative form $\ss$
on the domain $\dom \ss=\dom S$ via
\[
 \ss[f,g]=(f',g), \quad \{f,f'\},\, \{g,g'\}\in S.
\]
The form $\ss$ is closable, i.e., its closure $\overline{\ss}$ is a closed nonnegative form.
On the other hand, if $\st$ is a closed nonnegative form in a Hilbert space $\sH$, then
the first representation theorem asserts that there is a unique nonnegative selfadjoint
relation $H$ in $\sH$ such that $\st$ is the closure of the nonnegative form determined by $H$.
This one-to-one correspondence between closed nonnegative forms and nonnegative selfadjoint relations
in $\sH$ is indicated by $\st=\st_H$.
More precisely, $\st=\st_{H_{\rm s}}$, where $H_{\rm s}$ is the selfadjoint operator part of $H$ and
$\mul H=\sH\ominus\overline{\dom \st}$.

If $S$ is a nonnegative relation, then the closure of $\ss$ is a closed
nonnegative form $\st_{S_F}$ that corresponds to a nonnegative
selfadjoint extension $S_F$ of $S$, namely the Friedrichs extension of $S$.
Note that in the case that $S$ is selfadjoint, its so-called \textit{Friedrichs extension}
coincides with $S$. In general,
the Friedrichs extension $S_F$ of $S$ can be obtained by
\begin{equation}\label{frform}
 S_F=\{\, \{h,k\} \in S^*:\, h \in \dom \st_{S_F} \,\}.
\end{equation}
Since $S$ is nonnegative,  so is $S^{-1}$. Therefore, also
\begin{equation}\label{ando}
S_{K}=((S^{-1})_{F})^{-1}
\end{equation}
is a nonnegative selfadjoint extension of $S$, the so-called
\textit{Kre\u{\i}n-von Neumann extension}.
Thanks to \eqref{frform} (with $S$ replaced by $S^{-1}$)  and    \eqref{ando},
the Kre\u{\i}n-von Neumann extension $S_K$ of $S$ can be obtained by
\begin{equation}\label{krform}
 S_K=\{\, \{h,k\} \in S^*:\, k \in \dom \st_{(S^{-1})_F} \,\}.
\end{equation}

The Friedrichs extension and the Kre\u{\i}n-von Neumann extension
are extreme extensions in the following sense.
 If $A$ is nonnegative selfadjoint extension of $S$, then
 $S_K \leq A \leq S_K$, or, equivalently,
 \begin{equation}\label{kreinn}
(S_F+I)^{-1} \leq (A+I)^{-1} \leq (S_K+I)^{-1}.
\end{equation}
Conversely, if $A$ is a nonnegative selfadjoint relation
that satisfies \eqref{kreinn}, then $A$ is an extension, not only of $S$,
but also of the closed symmetric relation $S_0=S_F \cap S_K$ of $S$,
that is $S_0 \subset A$; cf. \cite[Theorem 5.4.6]{BHS}.
Consequently, the nonnegative selfadjoint extensions of $S$ and $S_0$ coincide.

Equivalent to the inequalities in \eqref{kreinn} is that the corresponding forms satisfy
\[
 \st_{S_K} \leq \st_A \leq \st_{S_F};
\]
cf. \cite{BHS}, where the last inequality actually means $\st_{S_F} \subset \st_A$.
A nonnegative selfadjoint extension $A$ of $S$ is said to be \textit{extremal} if
\begin{equation}\label{extr}
 (\st_{S_F} \subset )\,\, \st_{A} \subset \st_{S_K}.
\end{equation}
 It is known that a nonnegative selfadjoint extension $A$ of $S$ is extremal if and only if
\[
 \inf\,\big\{\,(f'-h',f-h):\, \{h,h'\} \in S\,\big\}=0 \quad \textrm{for all } \{f,f'\}\in A.
\]
cf.  \cite{ArTs88}.
For various equivalent conditions for extremality of $A$,
see also \cite{Ar3}, \cite{AHSS}, and further references in these papers.
 By the above definition, which uses the inclusion in $\st_{S_K}$ of the
associated closed forms, it is clear that the extremal extensions of $S$
are at the same time also extremal extensions of $S_0$ and, vice versa. \\

The case of present interest is where the numerical range of the symmetric relation
$S$ in $\sH$  is trivial: $\cW(S)=\{0\}$; see Section \ref{orth}.
Then the form $\ss$ determined by $S$ is trivial by Lemma \ref{numranzero}:
\[
 \ss[f,g]=(f',g)=0, \quad \{f,f'\},\, \{g,g'\}\in S.
\]
In particular, the form topology coincides with the Hilbert space topology.
Then the closure $\st_{S_F}$ of $\st_S$ satisfies
\[
\st_{S_F}=0, \quad \dom \st_{S_F}=\cdom S.
\]
Therefore, the Friedrichs extension $S_F$ of $S$ is given by
\begin{equation}\label{fr}
  S_F=\{\, \{h,k\} \in S^*:\, h \in \cdom S \,\};
\end{equation}
cf. \eqref{frform}. Likewise, since also $\cW(S^{-1})=\{0\}$,
it follows from \eqref{ando} that
\begin{equation}\label{kr}
  S_K=\{\, \{h,k\} \in S^*:\, k \in \cran S \,\}.
\end{equation}
Now let $A$ be a nonnegative selfadjoint extension of $S$ such that
$\cW(A)=\{0\}$, which clearly implies that $\cW(S)=\{0\}$.
Then the corresponding form $\st_A$ is trivial with closed
domain $\dom A$ that contains $\cdom S$.

\begin{lemma}\label{extremal0}
Let $S$ be nonnegative relation in a Hilbert space $\sH$ and assume that
$\cW(S_K)=\{0\}$. Then for a nonnegative selfadjoint extension
$A$ of $S$ the following conditions are equivalent:
\begin{enumerate} \def\labelenumi{\rm(\roman{enumi})}
\item $A$ is an extremal extension of $S$;
\item $\cW(A)=\{0\}$.
\end{enumerate}
\end{lemma}

\begin{proof}
The assumption about $S_K$ shows that $\dom S_K \perp\ran S_K$.
Hence the closed form $\st_{S_K}$ corresponding to $S_K$
is the zero form on the closed domain $\dom S_K$.

(i) $\Rightarrow$ (ii)
Let $A$ be an extremal extension of $S$. Then by \eqref{extr}
one has $\st_A \subset \st_{S_K}$. Hence $\st_A$ is the zero form
on $\dom \st_A$. In particular, it follows that  $\cW(A)=\{0\}$.

(ii) $\Rightarrow$ (i)
Assume that $\cW(A)=\{0\}$, so that
the closed form generated by $A$ is the zero form on its necessarily closed domain.
 By the inequality $S_K \leq A$ one has
 $\dom \st_{A}\subset \dom \st_{S_K}$ and hence as
a zero form $\st_{A}$ is a closed restriction of the form $\st_{S_K}$, i.e., it satisfies \eqref{extr}. Hence $A$ is an extremal extension of $S$.
\end{proof}

\section{Explicit description of all nonnegative selfadjoint extensions}\label{Fried}

This section contains formulas for the  Friedrichs and
Kre\u{\i}n-von Neumann extensions
of $S$ in \eqref{tstar0}. As, in general,
they are not transversal as extensions of $S$,
the closed symmetric extension $S_F \cap S_K$ of $S$
will be used as the underlying symmetric extension
for an alternative boundary triplet.
First, the Friedrichs extension $S_F$ of $S$ will be determined.

\begin{lemma}\label{friedrichs}
Let $R$ be a closed linear relation from $\sH_1$ to $\sH_2$ and let $S$
be the  relation defined in  \eqref{tstar0}.
Then the Friedrichs extension $S_{F}$ of $S$ is given by
\begin{equation}\label{friedfor}
S_{F} =( \cdom R \oplus  \{0\})^\top
 \times ( \mul R^{*} \oplus \sH_2)^\top.
 \end{equation}
\end{lemma}

\begin{proof}
Observe from the definition of $S$ in \eqref{tstar0} that  $\cW(S)=\{0\}$ and that
\[
 \cdom S= (\cdom R \oplus \{0\})^\top.
\]
Then, thanks to \eqref{fr}, one sees that
\[
S_F=\left\{ \left\{\begin{pmatrix} h_{1}  \\ h_{2} \end{pmatrix},
 \begin{pmatrix} k_{1}  \\ k_{2} \end{pmatrix} \right\} \in S^*:\,
  h_{1} \in \cdom R, \, h_{2}=0
   \, \right\}.
\]
Hence, it follows from \eqref{tstar} that \eqref{friedfor} holds.
\end{proof}

Next,  the Kre\u{\i}n-von Neumann extension $S_{K}$
will be determined in a similar way.

\begin{lemma}\label{krein}
Let $R$ be a closed linear relation from $\sH_1$ to $\sH_2$ and let $S$
be the   relation defined in  \eqref{tstar0}.
Then the Kre\u{\i}n-von Neumann extension $S_{K}$ of $S$ is given by
\begin{equation}\label{kreinfor}
S_{K} =(  \sH_1 \oplus \ker R^{*})^\top
 \times ( \{0\}  \oplus \cran R )^\top.
 \end{equation}
\end{lemma}

\begin{proof}
Observe from the definition of $S$ in \eqref{tstar0}
that $\cW(S^{-1})=\{0\}$ and
\[
 \cran S= ( \{0\} \oplus \cran R)^\top.
\]
Then, thanks to \eqref{kr}, one sees that
\[
S_K=\left\{ \left\{\begin{pmatrix} h_{1}  \\ h_{2} \end{pmatrix},
 \begin{pmatrix} k_{1}  \\ k_{2} \end{pmatrix} \right\} \in S^*:\,
 k_1=0, \, k_2 \in \cran R \, \right\}.
 \]
Hence, it follows from \eqref{tstar} that \eqref{kreinfor} holds.
\end{proof}

It is clear from Lemma \ref{krein} that $\dom S_{K}\perp \ran S_{K}$ or, equivalently,
$\cW(S_K)=\{0\}$; see Lemma \ref{numranzero}. Hence from Lemma \ref{extremal0}
one obtains the following characterization for extremal extensions of $S$.

\begin{corollary}\label{extremal}
Let $S$ be the relation defined in \eqref{tstar0}.
Then the Kre\u{\i}n-von Neumann extension $S_{K}$ of $S$
satisfies $\cW(S_K)=\{0\}$ and for a nonnegative selfadjoint extension
$A$ of $S$ the following conditions are equivalent:
\begin{enumerate} \def\labelenumi{\rm(\roman{enumi})}
\item $A$ is an extremal extension of $S$;
\item $\cW(A)=\{0\}$.
\end{enumerate}
\end{corollary}

The Friedrichs and the Kre\u{\i}n-von Neumann extensions are selfadjoint
extensions of $S$, which are both singular. According to Corollary \ref{write},
there are the block representations
\begin{equation}\label{friedforr}
S_F=\begin{pmatrix} \cdom R \times \mul R^* & \{0\} \times \mul R^* \\
\cdom R \times \sH_2 & \{0\} \times \sH_2 \end{pmatrix}
=\begin{pmatrix} \sH_1 \times \{0\} & \{0\} \times \mul R^* \\
\cdom R \times \sH_2 & \{0\} \times \sH_2 \end{pmatrix},
\end{equation}
cf. Remark \ref{mulmat},
and, likewise,
\[
S_{K} =\begin{pmatrix} \sH_1\times \{0\} & \ker R^{*}\times \{0\}  \\
 \sH_1 \times \cran R  & \ker R^* \times \cran R \end{pmatrix}.
\]
The Friedrichs and the Kre\u{\i}n-von Neumann extensions
have the same lower bound.
It may happen that the Friedrichs and
Kre\u{\i}n-von Neumann extensions of $S$ coincide.
The following statement is clear
from Lemma \ref{friedrichs} and Lemma \ref{krein}.

\begin{corollary}\label{PropFrasA0}
Let $R$ be a closed linear relation from $\sH_1$ to $\sH_2$
and let $S$ be the relation defined in  \eqref{tstar0}.
The following statements are equivalent:
\begin{enumerate} \def\labelenumi{\rm(\roman{enumi})}
\item $S_{F}=S_{K}$;
\item $\cdom R=\sH_1$ and $\cran R=\sH_2$.
\end{enumerate}
\end{corollary}

It follows from the above representations \eqref{friedfor}
and \eqref{kreinfor}
that the nonnegative selfadjoint extensions $S_F$ and
$S_K$ of $S$ satisfy
\[
 S_{F} \cap S_{K}=( \cdom R \oplus \{0\})^\top
 \times ( \{0\}  \oplus \cran R)^\top.
 \]
Thus  $S_{F}$ and $S_{K}$ are disjoint if and only if the
relation $R$ is singular. In the opposite case, $S_F$ and $S_K$ are
not disjoint and so not transversal.
Now introduce the following symmetric extension of $S$:
\begin{equation}\label{SFcapSK}
 S_0=S_{F} \cap S_{K}=( \cdom R \oplus \{0\})^\top
 \times ( \{0\}  \oplus \cran R)^\top.
 \end{equation}
Then, by definition, $S_F$ and $S_K$ are disjoint
as selfadjoint extensions of $S_0$.
It is known that the nonnegative selfadjoint extensions of $S$
and $S_0$ coincide; cf. Section \ref{Fried}.
The following lemma shows that $S_F$ and $S_K$ are transversal
extensions of $S_0$.

\begin{lemma}
The adjoint of the symmetric relation $S_0$
in \eqref{SFcapSK} is given by
\begin{equation}\label{S0star}
 S_0^*
 =\left\{\, \left\{\begin{pmatrix} h_{1}  \\ h_{2} \end{pmatrix},
 \begin{pmatrix} k_{1}  \\ k_{2} \end{pmatrix} \right\} :\,
 \begin{matrix} h_{1} \in \sH_{1}, &  k_{1}\in \mul R^* \\
h_{2} \in \ker R^* , & k_{2} \in \sH_{2} \end{matrix}\,  \right\}
\end{equation}
and it satisfies the equality $S_0^*=S_F \hplus S_K$.
\end{lemma}

\begin{proof}
The description of $S_0^*$ is obtained from \eqref{SFcapSK},
e.g., by means of the equality $S_0^*=J S_0^\perp$, which shows that
\[
 S_0^* =  ( \sH_1 \oplus \ker R^* )^\top
 \times ( \mul R^{*} \oplus \sH_2 )^\top;
 \]
cf. \eqref{decomp} and \eqref{decompl}.
The equality $S_0^*=S_F \hplus S_K$ is now clear from
the descriptions of $S_F$ in  \eqref{friedfor} and $S_K$ in \eqref{kreinfor}.
\end{proof}

 According to Corollary \ref{PropFrasA0} the equality $S_F=S_K$
 holds precisely when the subspace
\begin{equation}\label{G0}
 \cG_0=\mul R^* \times \ker R^* \subset \sH_1\times \sH_2
\end{equation}
is zero. In what follows it is assumed that $\cG_0\neq \{0\}$ and
all nonnegative selfadjoint extensions are described.
Observe, that $\cG_0\subset \cG=\sN_{-1}(S^*)$; see \eqref{hk00}.
First notice that for $\lambda \in \dC$ the eigenspace associated
with \eqref{S0star} is given by
\begin{equation}\label{hklS0}
 \wh \sN_{\lambda}(S_0^*)
 = \left\{ \,\left\{\begin{pmatrix} h_{1}  \\ h_{2} \end{pmatrix},
 \begin{pmatrix} k_{1}  \\ k_{2} \end{pmatrix} \right\}:\,
 \begin{matrix} k_{1}=\lambda h_{1}, &  h_{2} \in \ker R^* \\
 k_{2}=\lambda h_{2} , &   k_{1}\in \mul R^* \end{matrix} \,\right\}.
\end{equation}
In particular, for $\lambda\neq 0$ the eigenspace
$\sN_\lambda(S_0^*)=\ker (S_0^*-\lambda)$
has the form
\begin{equation}\label{hk0S0}
 \sN_{\lambda}(S_0^*)= \left\{ \begin{pmatrix} h_{1}  \\ h_{2} \end{pmatrix} :\,
 h_{1}\in \mul R^*, \, h_{2} \in \ker R^*  \,\right\}.
\end{equation}
Hence, $\sN_{\lambda}(S_0^*)=\cG_0\subset \cG$ for all $\lambda \neq 0$.
Let $Q_0$ be the orthogonal projection
from $\sH_1 \oplus \sH_2$ onto $\cG_0$,
i.e., $Q_0=P_{\mul R^*}\times P_{\ker R^*}$,
where $P_{\mul R^*}$ is the orthogonal projection
from $\sH_1$ onto $\mul R^*$ and
where $P_{\ker R^*}$ is the orthogonal projection
from $\sH_2$ onto $\ker R^*$. \\

In order to describe all nonnegative selfadjoint extensions of $S_0$,
it is convenient to construct a boundary triplet
$\{\cG_0,\Gamma^0_0,\Gamma^0_1\}$ for $S_0^*$
such that $S_F=\ker \Gamma^0_0$ and $S_K=\ker\Gamma^0_1$.
Such boundary triplets were introduced and studied by
Arlinski\u{\i} in \cite{Ar1} as a special case of so-called positive boundary triplets (also called positive boundary value spaces)
which were introduced earlier by Kochubei \cite{Koch} and used for describing nonnegative selfadjoint extensions of a nonnegative operator $S$
in the case when $0$ is a regular type point of $S$. The general case was treated also in \cite{DM87}.
A boundary triplet with $\ker \Gamma^0_0=S_F$ and $\ker\Gamma^0_1=S_K$ from \cite{Ar1} is often called a basic (positive) boundary triplet
(cf. \cite{AHSS}, \cite{BHS}).
Such a boundary triplet is convenient, since all nonnegative
selfadjoint extensions of $S_0$ can be parametrized
simply by means of nonnegative selfadjoint relations $\Theta$ in the (boundary) space $\cG_0$
(cf. Theorem \ref{BTS0op} below).

\begin{proposition}\label{BTS0}
Let the symmetric relation $S_0$ be defined by
\eqref{SFcapSK} with the adjoint \eqref{S0star}.
Let $Q_0$ be the orthogonal projection from
$\sH_1 \oplus \sH_2$ onto $\cG_0$. Then for
\[
 \left\{\begin{pmatrix} h_{1}  \\ h_{2} \end{pmatrix},
 \begin{pmatrix} k_{1}  \\ k_{2} \end{pmatrix} \right\} \in S_0^{*},
\]
define
\begin{equation}\label{btS0}
 \Gamma^0_0  \left\{\begin{pmatrix} h_{1}  \\ h_{2} \end{pmatrix},
 \begin{pmatrix} k_{1}  \\ k_{2} \end{pmatrix} \right\}
 =Q_0\begin{pmatrix} h_{1}  \\ h_{2} \end{pmatrix}
\quad \mbox{and} \quad
\Gamma^0_1  \left\{\begin{pmatrix} h_{1}  \\ h_{2} \end{pmatrix},
\begin{pmatrix} k_{1}  \\ k_{2} \end{pmatrix} \right\}
=Q_0  \begin{pmatrix} k_{1} \\ k_{2} \end{pmatrix}.
\end{equation}
Then $\{\cG_0, \Gamma^0_0, \Gamma^0_1\}$
is a boundary triplet for the relation $S_0^{*}$.
Furthermore, one has $\ker \Gamma^0_0=S_F$
and $\ker\Gamma^0_1=S_K$.
\end{proposition}

\begin{proof}
For general elements in $S_0^{*}$ of the form
\[
 \left\{\begin{pmatrix} h_{1}  \\ h_{2} \end{pmatrix},
 \begin{pmatrix} k_{1}  \\ k_{2} \end{pmatrix} \right\},
 \quad
 \left\{\begin{pmatrix} f_{1}  \\ f_{2} \end{pmatrix},
 \begin{pmatrix} g_{1}  \\ g_{2} \end{pmatrix} \right\},
\]
with $k_1,g_1\in\mul R^*$ and $h_2,f_2\in\ker R^*$
one has the Green identity
\[
\begin{split}
&\left( \begin{pmatrix} k_{1} \\ k_{2} \end{pmatrix} ,
\begin{pmatrix} f_{1}  \\ f_{2} \end{pmatrix} \right)
- \left( \begin{pmatrix} h_{1}  \\ h_{2} \end{pmatrix},
\begin{pmatrix} g_{1}  \\ g_{2} \end{pmatrix} \right)  \\
  &\hspace{1cm} =
                 (k_{1},f_{1}) +(k_{2},f_{2}) - (h_{1} , g_{1}) - (h_{2}, g_{2})  \\
 &\hspace{1cm} =
                (P_{\mul R^*}k_{1},f_{1}) +(k_{2},P_{\ker R^*}f_{2})
                - (h_{1} , P_{\mul R^*}g_{1}) - (P_{\ker R^*}h_{2}, g_{2})  \\
&\hspace{1cm} =
        \left( Q_0\begin{pmatrix} k_{1} \\ k_{2} \end{pmatrix} ,
        Q_0\begin{pmatrix} f_{1}  \\ f_{2} \end{pmatrix} \right)
       - \left( Q_0\begin{pmatrix} h_{1}  \\ h_{2} \end{pmatrix},
       Q_0\begin{pmatrix} g_{1}  \\ g_{2} \end{pmatrix} \right).
\end{split}
\]
Thus the abstract Green identity holds with the mappings
$\Gamma^0_{0}$ and $\Gamma^0_{1}$ in \eqref{btS0}.

Furthermore, in the definition of $S_0^{*}$ the elements
$h_{1} \in \sH_1$ and $h_{2} \in \ker R^*$ are arbitrary
and independent from the choice
of the elements $k_{1}\in\mul R^*$ and $k_{2}\in \sH_2$.
Hence, the pair of mappings
$(\Gamma^0_{0},\Gamma^0_{1})$ takes $S_0^{*}$
onto $\cG_0\times\cG_0$.
Consequently, $\{ \cG_0, \Gamma^0_0, \Gamma^0_1\}$
is a boundary triplet for $S_0^*$.

The identities $\ker \Gamma^0_0=S_F$ and $\ker\Gamma^0_1=S_K$
follow from the definitions in \eqref{btS0} and the descriptions of $S_F$
in \eqref{friedfor} and $S_K$ in \eqref{kreinfor}, respectively.
\end{proof}

The next result gives the $\gamma$-field and the Weyl function
corresponding to the boundary triplet
$\{\cG_0, \Gamma^0_0, \Gamma^0_1\}$.

\begin{proposition}
Let the boundary triplet $\{ \cG_0, \Gamma^0_{0}, \Gamma^0_{1}\}$
for $S_0^*$ be as defined in Proposition~{\rm \ref{BTS0}}.
Then the corresponding $\gamma$-field
and Weyl function are given by
\[
 \gamma_0(\lambda): \cG_0\to \sN_\lambda(S_0^*), \;
 \begin{pmatrix} h_{1}  \\ h_{2} \end{pmatrix} \to
 \begin{pmatrix} h_{1}  \\ h_{2} \end{pmatrix};
 \quad
 M_0(\lambda) =\lambda I_{\cG_0}, \quad \lambda \in \dC\setminus\{0\}.
\]
\end{proposition}

\begin{proof}
Recall from \eqref{hklS0} that for any $\lambda \neq 0$ one has that
\[
 \wh \sN_{\lambda}(S_0^*)
 = \left\{ \,\left\{\begin{pmatrix} h_{1}  \\ h_{2} \end{pmatrix},
 \lambda \begin{pmatrix} h_{1}  \\ h_{2} \end{pmatrix} \right\}:\,
 h_{1}\in \mul R^*, \,  h_{2} \in \ker R^*  \,\right\}.
\]
Thus, for the elements in $\wh \sN_{\lambda}(S_0^*)$
it follows from \eqref{btS0} and the equality
$\sN_{\lambda}(S_0^*)=\cG_0$, $\lambda\neq 0$, in \eqref{hk0S0} that
\[
 \Gamma_{0} \left\{\begin{pmatrix} h_{1}  \\ h_{2} \end{pmatrix},
  \lambda \begin{pmatrix} h_{1}  \\ h_{2} \end{pmatrix} \right\}
 =\begin{pmatrix} h_{1}  \\ h_{2} \end{pmatrix},
  \quad
  \Gamma_{1} \left\{\begin{pmatrix} h_{1}  \\ h_{2} \end{pmatrix},
 \lambda \begin{pmatrix} h_{1}  \\ h_{2} \end{pmatrix} \right\}
 =\lambda \begin{pmatrix} h_{1} \\ h_{2} \end{pmatrix}.
\]
Therefore, by definition, the graph of the Weyl function $M_0$ is given by
\[
 M_0(\lambda)=\left\{ \left\{ \begin{pmatrix} h_{1} \\ h_{2} \end{pmatrix},
 \lambda \begin{pmatrix} h_{1} \\ \lambda h_{2} \end{pmatrix} \right\} :\,
 h_{1}\in\mul R^*, h_{2}\in\ker R^{*}\,\right\},
\]
i.e., $M_0(\lambda)=\lambda I_{\cG_0}$.

Likewise, by definition, the graph of the $\gamma$-field is given by
\[
 \gamma_0(\lambda)
 =\left\{ \left\{ \begin{pmatrix} h_{1} \\ h_{2} \end{pmatrix},
 \begin{pmatrix} h_{1} \\ h_{2} \end{pmatrix} \right\} :\,
 h_{1}\in\mul R^*, h_{2}\in\ker R^{*} \,\right\},
\]
so that $\gamma_0(\lambda)$ is a constant (inclusion) mapping
from $\cG_0$ onto $\sN_{\lambda}(S_0^*)$, $\lambda\neq 0$.
\end{proof}

It is possible to describe all nonnegative selfadjoint extensions of $S$
in an explicit form by means
of suitable block relation formulas.
For this purpose, first notice that
\[
 \mul S_0=(\{0\} \oplus \cran R)^\top.
\]
Hence $S_0$ can be decomposed via its operator part $(S_0)_{\rm op}$ as follows
 \begin{equation}\label{S0decom}
 S_0= (S_0)_{\rm op} \hoplus (S_0)_\mul
 = (S_0)_{\rm op} \hoplus \left( \{0\}\times (\{0\} \oplus \cran R)^\top \right),
\end{equation}
where $(S_0)_\mul=\{0\}\times \mul S_0$ is a selfadjoint relation in $\cran R$
which appears as an orthogonal selfadjoint part in the adjoint of $S_0$
as well as in every selfadjoint extension of $S_0$ in $\sH_1\oplus\sH_2$.
Therefore, it suffices to consider the selfadjoint extensions of
the operator part $(S_0)_{\rm op}$ in the closed subspace
\[
 \sH_0:=\sH_1\oplus \ker R^*.
\]
Observe that
\[
(S_0)_{\rm op}= 0\uphar \dom S_0= 0\uphar \cdom R=\cdom R \times \{0\}.
\]
The adjoint of $(S_0)_{\rm op}$ in $\sH_0$ is given by
\begin{equation}\label{S0op*}
 ((S_0)_{\rm op})^* = (\sH_1\oplus \ker R^*)^\top\times ( \mul R^{*} \oplus \ker R^*)^\top= \sH_0 \times \cG_0,
\end{equation}
see \eqref{G0}. It is natural to decompose $\sH_0$ as follows
\[
 \sH_0
 =\cdom R \oplus (\mul R^* \oplus \ker R^*)= \cdom R \oplus \cG_0.
\]
Now the following result is obtained from Proposition \ref{BTS0}
after restricting the mappings $\Gamma^0_0$ and $\Gamma^0_1$
therein to $((S_0)_{\rm op})^*$; for simplicity the same notation
is kept here for these two restrictions; see \cite[Remark 2.3.10]{BHS}.

\begin{theorem}\label{BTS0op}
Let the symmetric relation $(S_0)_{\rm op}$ be the operator part of $S_0$ in the subspace $\sH_0=\sH_1\oplus \ker R^*$
with the adjoint \eqref{S0op*}.
Let $Q^0_0$ be the orthogonal projection from $\sH_0$ onto $\cG_0$. Then for an element
\[
 \left\{\begin{pmatrix} h_{1}  \\ h_{2} \end{pmatrix},
 \begin{pmatrix} k_{1}  \\ k_{2} \end{pmatrix} \right\} \in ((S_0)_{\rm op})^{*},
\]
define
\begin{equation}\label{btS0op}
 \Gamma^0_0  \left\{\begin{pmatrix} h_{1}  \\ h_{2} \end{pmatrix},
 \begin{pmatrix} k_{1}  \\ k_{2} \end{pmatrix} \right\}
 =Q^0_0\begin{pmatrix} h_{1}  \\ h_{2} \end{pmatrix}
\quad \mbox{and} \quad
\Gamma^0_1  \left\{\begin{pmatrix} h_{1}  \\ h_{2} \end{pmatrix},
\begin{pmatrix} k_{1}  \\ k_{2} \end{pmatrix} \right\}
=Q^0_0  \begin{pmatrix} k_{1} \\ k_{2} \end{pmatrix}.
\end{equation}
Then $\{\cG_0, \Gamma^0_0, \Gamma^0_1\}$
is a boundary triplet for the adjoint $((S_0)_{\rm op})^{*}$.
Furthermore, the (nonnegative) selfadjoint extensions $S_\Theta$ of $(S_0)_{\rm op}$ in $\sH_0$ are
in one-to-one correspondence with the (nonnegative) selfadjoint relations $\Theta$ in $\cG_0$ via
\begin{equation}\label{S0Block}
 S_\Theta = O_{\cdom R} \hoplus \Theta,
 \end{equation}
where the decomposition is according to $\sH_0=\cdom R \oplus \cG_0$.

In particular, the extremal extensions $S_\Theta$ of $(S_0)_{\rm op}$ are in one-to-one correspondence
with the closed subspaces $\sL\subset \cG_0$ via $\Theta=\sL\times (\cG_0\ominus\sL)$.
\end{theorem}

\begin{proof}
First notice that the component $(S_0)_\mul$ of $S_0$ in \eqref{S0decom} belongs to the intersection
$\ker\Gamma^0_0\cap \ker \Gamma^0_1$. Moreover, since $S_0^*=((S_0)_{\rm op})^* \hoplus (S_0)_\mul$, it is clear that
by restricting the mappings $\Gamma^0_0$ and $\Gamma^0_1$ to $((S_0)_{\rm op})^*$, one obtains from the boundary triplet
for $S_0^*$ a boundary triplet for $((S_0)_{\rm op})^*$ as defined in \eqref{btS0op}.

Next observe that since $\sH_0=\dom (S_0)_{\rm op}\oplus \cG_0$ and
$(S_0)_{\rm op}=\cdom R \times \{0\}$ while $((S_0)_{\rm op})^*=\sH_0 \times \cG_0$, see \eqref{S0op*},
one has the following orthogonal componentwise decomposition:
\[
 ((S_0)_{\rm op})^*= (S_0)_{\rm op} \hoplus (\cG_0\times \cG_0).
\]
Therefore, by decomposing $\wh f\in ((S_0)_{\rm op})^*$
according to this decomposition in the form
$\wh f=\wh f_0 \hoplus \wh f_{\cG_0}$
with $\wh f_0\in (S_0)_{\rm op}$ and
$\wh f_{\cG_0}=\{f_{\cG_0},f_{\cG_0}^\prime\}\in \cG_0\times \cG_0$,
it follows that
\[
 \Gamma^0_0\wh f=\Gamma^0_0\wh f_{\cG_0}
 =Q^0_0 f_{\cG_0}= f_{\cG_0}, \quad
 \Gamma^0_1\wh f=\Gamma^0_1\wh f_{\cG_0}
 =Q^0_0 f_{\cG_0}^\prime= f_{\cG_0}^\prime.
\]
Hence, the pair of mappings $(\Gamma^0_0,\Gamma^0_1)$
act as the identity mapping on the component $\cG_0\times\cG_0$
and vanishes on the other component $(S_0)_{\rm op}$ of $((S_0)_{\rm op})^*$.
This proves the explicit block formula \eqref{S0Block} for
the selfadjoint extensions of $(S_0)_{\rm op}$.

Due to \eqref{S0Block},  the Kre\u{\i}n extension of $(S_0)_{\rm op}$
corresponds to $\Theta=\sH_0\times \{0\}$.
The corresponding form $\st_K$ is just the
zero form on the domain $\dom \st_K=\sH_0$.
Since extremal extensions are the
nonnegative selfadjoint extensions whose associated closed forms
are restrictions of the form
$\st_K$, they are zero forms on the closed subspaces
$\dom S_0\oplus\sL$, where $\sL\subset\cG_0$.
This clearly implies the formula for the selfadjoint relations
associated to such closed forms
and completes the proof.
\end{proof}

Note that the (nonnegative) selfadjoint extension $S_\Theta$
of $(S_0)_{\rm op}$ in $\sH_0$ can be written as a block relation
\[
 S_\Theta = \begin{pmatrix}
                  0 & 0 \\
                  0 & \Theta
                \end{pmatrix},
\]
involving the relation $\Theta$.
Such block representations for selfadjoint extensions of a bounded operator can be found in \cite[Proposition~5.1]{HMS2004},
where a different boundary triplet was used; see also \cite[Remark 2.4.4]{BHS}.
It is possible to obtain a connection to the boundary triplet in \cite{HMS2004} by using
the following expression for the adjoint of $(S_0)_{\rm op}$:
\[
  ((S_0)_{\rm op})^* =S_K \hoplus \left( \{0\}\times \cG_0 \right).
\]

Notice that the extremal extensions described in Theorem \ref{BTS0op}
correspond to the boundary conditions
in $\cG_0$ that are determined by the orthogonal projections
$P_\sL$ from $\cG_0$ onto $\sL$; cf. \cite[Proposition 7.1]{AHSS}.
Recall that orthogonal projections $P_\sL$ are extreme points
of the operator interval $[0,I_{\cG_0}]$,
which also motivates the term ``extremal extension'' in this situation.
There are further descriptions of extremal extensions.
In particular, \cite[Theorem 8.3]{AHSS} contains a purely analytic description
of extremal extensions by means of associated Weyl functions.
In the present situation this would
lead to the following analytic description: the Weyl functions
(of appropriately transformed boundary triplets)
of all extremal extensions are of the form:
\[
 M_\Theta(\lambda)= -1/\lambda I_\sL \oplus \lambda I_{\cG_0\ominus\sL}.
\]

\section{Semibounded extensions and associated semibounded parameters}\label{special}

In this section semibounded selfadjoint extensions of $S$
are investigated.
For this purpose it is convenient to introduce
a symmetric extension $\wt S$ of $S$
by reducing the parameter space $\cG$ slightly,
in case the original relation $R$ is not densely defined.
The corresponding boundary triplet
has a parameter space $\wt \cG \subset \cG$
and due this restriction the corresponding Weyl function
has a specific asymptotic behavior. \\

Assume that the linear relation $R$ from $\sH_1$ to $\sH_2$ is closed
and let the symmetric relation $S$ in $\sH=\sH_1 \oplus \sH_2$ be as in
\eqref{tstar0}.  Define the linear relation $\wt S$ by
\begin{equation}\label{tstar0w}
\wt S
 =\left\{\, \left\{\begin{pmatrix} f_{1}  \\ 0 \end{pmatrix},
 \begin{pmatrix} g_{1}  \\ g_{2} \end{pmatrix} \right\} :
 \, \{f_{1}, g_2\} \in R, \, g_{1} \in \mul R^* \,  \right\},
\end{equation}
so that
\begin{equation}\label{tstar0w1}
\wt S=S \hplus (\{0\} \oplus \{0\})^\top \times  (\mul R^* \oplus \{0\})^\top. 
\end{equation}
Note that $\dom \wt S \perp \ran \wt S$ and that
$\wt S$ is a closed symmetric extension of $S$.
 It follows from \eqref{tstar0w1}, together with \eqref{tstar},   that
\begin{equation}\label{tstarw}
(\wt S)^*
 = \left\{\, \left\{\begin{pmatrix} h_{1}  \\ h_2 \end{pmatrix},
 \begin{pmatrix} k_{1}  \\ k_{2} \end{pmatrix} \right\} :\,
 h_1 \in \cdom R, \, \{h_{2}, k_1\} \in R^*, \, k_{2} \in \sH_2 \,  \right\}.
\end{equation}
Observe that matrix representations for $\wt S$ and $(\wt S)^*$ are given by
\begin{equation}\label{symm0*m}
\wt S=\begin{pmatrix} \sH_1  \times \{0\} &\{0\} \times  \mul R^{*} \\
R  & \{0\} \times \{0\} \end{pmatrix}, \quad
 (\wt S)^{*}=
\begin{pmatrix} \sH_1 \times \{0\} & R^{*} \\
\cdom R \times \sH_2  & \sH_2 \times \sH_2 \end{pmatrix}.
\end{equation}
For $\lambda \in \dC$ the eigenspace associated with \eqref{tstarw} is given by
\begin{equation}\label{hklllw}
 \wh \sN_{\lambda}((\wt S)^*)
 = \left\{ \,\left\{\begin{pmatrix} h_{1}  \\ h_{2} \end{pmatrix},
 \begin{pmatrix} k_{1}  \\ k_{2} \end{pmatrix} \right\}:\,
 \begin{matrix} h_1 \in \cdom R, & k_{1}=\lambda h_{1} \\
  k_{2}=\lambda h_{2}, & \{h_{2}, k_{1}\} \in R^* \end{matrix} \,\right\},
\end{equation}
and,  hence, with $\sN_\lambda((\wt S)^*)=\ker ((\wt S)^*-\lambda)$, one has
\[
 \sN_{\lambda}((\wt S)^*)= \left\{ \begin{pmatrix} h_{1}  \\ h_{2} \end{pmatrix} :
 \,\{h_{2},\lambda h_{1}\} \in (R^*)_{\rm s} \,\right\}.
\]
Since $\dom \wt S=\dom S$, one sees that
 \[
 \mul (\wt S)^*= \mul S^*=(\mul R^* \oplus \sH_2)^\top.
 \]
Similar to the  situation in Section \ref{BT5}, an eigenspace of
$(\wt S)^*$
will play a special role:
\begin{equation}\label{gw}
 \wt \cG=\sN_{-1}((\wt S)^*)
 =\left\{ \begin{pmatrix} h_{1}  \\ h_{2} \end{pmatrix} :
 \,\{h_{2},- h_{1}\} \in (R^*)_{\rm s} \,\right\}= J(R^*)_{\rm s}.
\end{equation}
It is straightforward to see that (cf. \eqref{decomp}, \eqref{decompl})
 \begin{equation}\label{orthw}
 \sH_1 \oplus \sH_2
 =R \hoplus (\mul R^* \oplus \{0\}) \hoplus J\,(R^*)_{\rm s}
 =  R \hoplus (\mul R^* \oplus \{0\}) \hoplus \wt \cG.
\end{equation}

\begin{proposition}\label{PropFrasA0+}
Let $R$ be a closed linear relation from $\sH_1$ to
$\sH_2$, let $\wt S$ be defined by \eqref{tstar0w}
with adjoint \eqref{tstarw}, and let $\wt Q$ be the orthogonal projection from
$\sH_1 \oplus \sH_2$ onto $\wt\cG$ in \eqref{gw}. With an element
\begin{equation}\label{hkw}
 \left\{\begin{pmatrix} h_{1}  \\ h_{2} \end{pmatrix},
 \begin{pmatrix} k_{1}  \\ k_{2} \end{pmatrix} \right\},
 \quad h_1 \in \cdom R, \,\{h_{2}, k_{1}\} \in R^{*},  \, k_2 \in \sH_2,
\end{equation}
in $(\wt S)^{*}$ define
\begin{equation}\label{btw}
\wt  \Gamma_0  \left\{\begin{pmatrix} h_{1}  \\ h_{2} \end{pmatrix},
 \begin{pmatrix} k_{1}  \\ k_{2} \end{pmatrix} \right\}
 =\wt Q \begin{pmatrix} -k_{1}  \\ h_{2} \end{pmatrix}
\quad \mbox{and} \quad
\wt \Gamma_1  \left\{\begin{pmatrix} h_{1}  \\ h_{2} \end{pmatrix},
\begin{pmatrix} k_{1}  \\ k_{2} \end{pmatrix} \right\}
=\wt Q  \begin{pmatrix} h_{1} \\ k_{2} \end{pmatrix}.
\end{equation}
Then $\{ \wt \cG, \wt \Gamma_0, \wt \Gamma_1\}$
is a boundary triplet for $(\wt S)^{*}$  such that
\begin{equation}\label{hhw}
 \ker \wt \Gamma_0= S_F,
\end{equation}
where $S_F$ is given by \eqref{friedfor}, and
\begin{equation}\label{kkw}
\ker \wt \Gamma_1=K,
\end{equation}
where $K$ is given by \eqref{kk}.
Moreover, the corresponding Weyl function $\wt M(\lambda)\in \bB(\wt \cG)$  is given by
 \begin{equation}\label{Weyltilde}
\wt M(\lambda)
 =\wt Q \begin{pmatrix} -\frac{1}{\lambda} & 0\\
 0 & \lambda \end{pmatrix}
 \uphar {\wt \cG},
 \quad \lambda \in \dC \setminus \{0\}.
 \end{equation}
\end{proposition}

\begin{proof}
The fact that $\{ \wt \cG, \wt \Gamma_0, \wt \Gamma_1\}$
is a boundary triplet for $(\wt S)^{*}$ can be proved as in Theorem \ref{BT}.
To get the formula for the Weyl function $\wt M(\lambda)$
apply \eqref{btw} to the elements in \eqref{hklllw} to obtain
\[
 \wt M(\lambda)=\left\{
 \left\{ \wt Q \begin{pmatrix} -\lambda h_{1} \\ h_{2} \end{pmatrix},
 \wt Q \begin{pmatrix} h_{1} \\ \lambda h_{2} \end{pmatrix} \right\} :\,
 \{h_{2}, \lambda h_{1} \} \in (R^{*})_{\rm s}\,\right\}.
\]
Here the first entry belongs to $\wt\cG$ due to
$\{h_{2}, \lambda h_{1} \} \in (R^{*})_{\rm s}$ and
this leads to \eqref{Weyltilde} as in the proof of Theorem \ref{WEYL}.

To see the identity \eqref{hhw}, note that the element in \eqref{hkw} belongs
to $\ker \wt \Gamma_0$ if and only if
\[
 \wt Q \begin{pmatrix} -k_1 \\ h_2 \end{pmatrix}=0.
\]
It follows from \eqref{orthw} that this is the case precisely  if
\[
 h_2=0 \quad \mbox{and} \quad k_1 \in \mul R^*,
\]
and, consequently, one sees from \eqref{tstarw} that
\begin{equation*}
 \ker \wt \Gamma_0
=\left\{\, \left\{  \begin{pmatrix} h_{1} \\ 0\end{pmatrix},
 \begin{pmatrix} k_1  \\ k_{2} \end{pmatrix} \right\} :
 \, h_{1} \in \cdom R, \, k_1 \in \mul R^*, \, k_{2} \in \sH_2\, \right\}.
\end{equation*}
Comparison with Lemma \ref{friedrichs} shows that this extension equals
 the Friedrichs extension $S_F$ of $S$.
Likewise, to see the identity \eqref{kkw},
note that the element in \eqref{hkw} belongs
to $\ker \wt \Gamma_1$ if and only if
 \[
 \wt Q  \begin{pmatrix} h_{1} \\ k_{2} \end{pmatrix}=0.
\]
Thanks to \eqref{orthw}, this is the case precisely if
\begin{equation}\label{eqRR}
 \begin{pmatrix} h_{1} \\ k_{2} \end{pmatrix}
 \in R \oplus (\mul R^* \oplus \{0\})
 \quad \Leftrightarrow \quad
 \{h_1, k_2 \} \in R,
 \end{equation}
and this equivalence confirms \eqref{kkw}.
As to \eqref{eqRR} it suffices to check the implication $(\Rightarrow)$.
By assumption, there exists an element  $\varphi \in \mul R^*$, such that
\[
 \{h_1 +\varphi, k_2\} \in R.
\]
In particular, $h_1 +\varphi \in \dom R$, while by definition $h_1 \in \cdom R$
 (cf. \eqref{hkw}).
Thus $\varphi \in \cdom R$ which, together with $\varphi \in \mul R^*$,
implies that $\varphi=0$.
\end{proof}

Next the Friedrichs and Kre\u{\i}n-von Neumann extensions of $\wt S$
will be determined via \eqref{frform} and \eqref{krform}.

\begin{lemma}\label{friedrichsw}
Let $R$ be a closed linear relation from $\sH_1$ to $\sH_2$ and let $\wt S$
be the  relation defined in  \eqref{tstar0w}.
The Friedrichs extension $S_{F}$ of $S$ is given by
\begin{equation}\label{friedforw}
\wt S_{F} = (\cdom R \oplus \{0\})^\top \times (\mul R^* \oplus \sH_2)^\top
  =S_F.
\end{equation}
\end{lemma}

\begin{proof}
Observe from the definition of $S$ in \eqref{tstar0w}
that  $\cW(\wt S)=\{0\}$ and that
 \[
\cdom \wt S=
 (\cdom R \oplus \{0\})^\top.
\]
Then, thanks to \eqref{fr}, one sees that
\[
\wt S_F= \left\{\, \left\{\begin{pmatrix} h_{1}  \\ h_2 \end{pmatrix},
 \begin{pmatrix} k_{1}  \\ k_{2} \end{pmatrix} \right\} \in \wt S^* :\,
  h_{1}  \in \cdom R, \,  h_2=0
  \,  \right\}.
\]
Hence, it follows from \eqref{tstarw} that \eqref{friedforw} holds.
\end{proof}

\begin{lemma}
Let $R$ be a closed linear relation from $\sH_1$ to $\sH_2$ and let $\wt S$
be the   relation defined in  \eqref{tstar0w}.
The Kre\u{\i}n-von Neumann extension $\wt S_{K}$ of $\wt S$ is given by
\begin{equation}\label{kreinforw}
\wt S_{K} =( \cdom R  \oplus \ker R^* )^\top \times (\mul R^{*} \oplus  \cran R )^\top.
  \end{equation}
 \end{lemma}

\begin{proof}
Observe that $\cW(\wt S^{-1})=\{0\}$ and that
\[
 \cran \wt S = (\mul R^* \oplus \cran R)^\top.
\]
Thanks to \eqref{fr} one sees
 \begin{equation*}
\wt S_K =
\left\{
 \left\{\begin{pmatrix} h_{1}  \\ h_{2} \end{pmatrix},
 \begin{pmatrix} k_{1}  \\ k_{2} \end{pmatrix} \right\} \in (\wt S)^*:\,
   k_1 \in   \mul R^*,\, k_2 \in \cran R
   \,\right\}.
 \end{equation*}
Hence, it follows from \eqref{tstarw}  that \eqref{kreinforw} holds.
\end{proof}

Notice that $\dom \wt S_K \perp \ran \wt S_K$, so that $\cW(\wt S_K)=\{0\}$ and thus $\wt A=\wt A^*\geq 0$
is an extremal extension of $\wt S$ if and only if $\cW(\wt A)=\{0\}$; see Lemma \ref{extremal0}.

Recall that $\ker \Gamma_0$ in Theorem \ref{BT}
is the nonnegative selfadjoint extension $H$ as given in \eqref{hh},
while $\ker \wt\Gamma_0$ in Proposition \ref{PropFrasA0+}
is the Friedrichs extension of $S$ and $\wt S$.
In particular, $H\leq S_F$ and here equality $H=S_F$ holds
if and only if $R$ is densely defined in $\sH_1$ or,
equivalently, $R^*$ is an operator from $\sH_2$ to $\sH_1$.
In this case $\wt S=S$ and the boundary triplet in
Proposition \ref{PropFrasA0+} coincides
with the one in Theorem \ref{BT}. \\

For the block representations of the Friedrichs and Kre\u{\i}n-von Neumann
extensions, note that in terms of block representations one has $\wt S_F=S_F$
as given in \eqref{friedforr}.  It follows from \eqref{kreinforw}
and Corollary \ref{write} that
\[
\begin{split}
 \wt S_{K} & =\begin{pmatrix} \cdom R  \times \mul R^* & \ker R^* \times \mul R^*\\
 \cdom R \times \cran R &  \ker R^{*} \times \cran R \end{pmatrix} \\
 &=\begin{pmatrix} \sH_1 \times \{0\} & \ker R^* \times \mul R^*\\
 \cdom R \times \cran R &  \ker R^{*} \times \cran R \end{pmatrix},
\end{split}
\]
cf. Remark \ref{mulmat} and  \eqref{symm0*m}. \\

Observe that the Weyl function $M(\lambda)\in \bB(\cG)$
in Theorem \ref{WEYL} has the
following limit behavior
\[
 \lim_{x\downarrow -\infty}\left( M(x)  \begin{pmatrix} h_{1} \\ 0 \end{pmatrix},
  \begin{pmatrix} h_{1} \\ 0 \end{pmatrix} \right)
 = 0, \quad   \begin{pmatrix} h_{1} \\  0  \end{pmatrix}  \in  \cG ,
\]
which is possible when $h_1 \in \mul R^*$.
The Weyl function $\wt M(\lambda)\in \bB(\wt \cG)$
in Proposition \ref{PropFrasA0+}
admits the same form as the Weyl function
$M(\lambda)\in \bB(\cG)$ in Theorem \ref{WEYL},
but acts in the smaller space $\wt \cG\subset\cG$;
cf.\eqref{hk00}, \eqref{gw}.
In fact, $\wt M(\lambda)$ is a compression of
$M(\lambda)$ to the subspace $\wt\cG$.
Hence, as in Corollary \ref{cor4.4}, $\wt M(\lambda)$
satisfies the following weak identity
\begin{equation}\label{weylw}
 \left( \wt M(\lambda)  \begin{pmatrix} h_{1} \\ h_{2} \end{pmatrix},
  \begin{pmatrix} h_{1} \\ h_{2} \end{pmatrix} \right)
=  -\frac{1}{\lambda} \,(h_{1}, h_{1})+ \lambda \,(h_{2}, h_{2}),
\quad   \begin{pmatrix} h_{1} \\  h_{2}  \end{pmatrix} \in \wt \cG,
\end{equation}
where $\lambda \in \dC \setminus \{0\}$.
This leads to an interesting limit result.
In fact, it is known that the limit property \eqref{weylw2} of the Weyl function
characterizes $\ker \wt\Gamma_0$ as the Friedrichs extension;
 see e.g. \cite[Corollary 4.1]{DM1995}.

\begin{lemma}
Let $R$ be a closed linear relation from $\sH_1$ to $\sH_2$
 and let $\{ \wt \cG, \wt \Gamma_0, \wt \Gamma_1\}$
be the boundary triplet for $(\wt S)^*$ with the Weyl function
$\wt M(\lambda)$ as in Proposition {\rm \ref{PropFrasA0+}}.
Then
\begin{equation}\label{weylw2}
 \lim_{x\downarrow -\infty}\left( \wt M(x)  \begin{pmatrix} h_{1} \\ h_{2} \end{pmatrix},
  \begin{pmatrix} h_{1} \\ h_{2} \end{pmatrix} \right)
 = -\infty, \quad   \begin{pmatrix} h_{1} \\  h_{2}  \end{pmatrix}  \in \wt \cG \setminus \{0,0\}.
\end{equation}
\end{lemma}

\begin{proof}
Consider the identity \eqref{weylw2} for $\lambda <0$, $ \lambda \to -\infty$,
and recall that
\[
\wt \cG =\left\{ \begin{pmatrix} h_{1}  \\ h_{2} \end{pmatrix} :
 \,\{h_{2},- h_{1}\} \in (R^*)_{\rm s} \,\right\}.
\]
Hence, if $h\in\wt \cG$ satisfies $h_2=0$,
then it follows that  $h_1=0$.
This gives a contradiction, thus $h_2\neq 0$ and,
therefore, \eqref{weylw2} holds.
\end{proof}

First recall the following general equivalence.
Let $S$ be a nonnegative relation and let
 $\{\cG,\Gamma_0,\Gamma_1\}$ be a boundary triplet for $S^*$
with $\ker\Gamma_0=S_F$,
where $S_F$ is the Friedrichs extension of $S$.
Let $A_\Theta$ be a selfadjoint extension of $S$ as in
\eqref{btbtnew}. Then the following implication for $x < 0$
is satisfied:
\begin{equation}\label{semieq}
 x \leq A_\Theta \quad \Leftrightarrow \quad M(x) \leq \Theta,
\end{equation}
 see \cite{DM1991}, \cite[Proposition~5.5.6]{BHS}.
In particular, this implies that if $A_\Theta$ is bounded from below,
then also $\Theta$ is bounded from below,
since $M(x)$ is a bounded operator for each $x<0$.
The converse statement does not hold in general;
see \cite[Theorem 3]{DM1991}, \cite[Proposition 4.4]{DM1995}
for a criterion which uses the uniform convergence
of the associated Weyl function $M(x)$ as $x \to -\infty$,
and \cite[Lemma 5.5.7]{BHS}.  \\

Now return to the symmetric relation $\wt S$ in \eqref{tstar0w}.
It follows from Lemma \ref{friedrichsw} that $\ker \wt\Gamma_0=\wt S_F$ 
and hence \eqref{semieq} can be applied to the boundary triplet
$\{ \wt \cG, \wt \Gamma_0, \wt \Gamma_1\}$
and the Weyl funtion $\wt M(\lambda)$ in Proposition {\rm \ref{PropFrasA0+}}.
In the following the notation
$\wt A_\Theta=\ker (\wt \Gamma_1 -\Theta \wt \Gamma_0)$
with $\Theta$ a linear relation in $\wt \cG$,
 will be used for an extension of $\wt S$.
The preservation of semiboundedness
in this boundary triplet depends essentially on the initial relation $R$.

\begin{theorem}\label{propsemibd}
Let $R$ be a closed linear relation from $\sH_1$ to $\sH_2$
 and let $\{ \wt \cG, \wt \Gamma_0, \wt \Gamma_1\}$
be the boundary triplet for $(\wt S)^*$ with the Weyl funtion
$\wt M(\lambda)$ as in Proposition {\rm \ref{PropFrasA0+}}.
Then the following alternative holds:
\begin{enumerate} \def\labelenumi{\rm(\roman{enumi})}
\item if $(R^*)_{\rm s}$ is a bounded operator,
then the selfadjoint extension $\wt A_\Theta$ of $\wt S$
is semibounded from below
if and only if  $\Theta$ is semibounded from below in $\wt\cG$;

\item if the operator $(R^*)_{\rm s}$ is unbounded,
then there are nonzero bounded operators $\Theta$
in $\wt\cG$ with arbitrary small
operator norm $\|\Theta\|$ such that the extension
$\wt A_\Theta$ is not semibounded from below.
\end{enumerate}
\end{theorem}

\begin{proof}
(i) Assume that $(R^*)_{\rm s}$ is bounded.
It suffices to prove that if $\Theta$ is semibounded from below,
then so is the selfadjoint extension $\wt S_\Theta$.
Recall from \eqref{gw}  that $h=\{h_1,h_2\}\in\wt\cG$ is equivalent to
$\{h_2,h_1\}\in -(R^*)_{\rm s}$;  thus, by assumption,
 $\|h_1\|\leq M \|h_2\|$ for some $0\leq M<\infty$.
Now consider the values of the Weyl function $\wt M(x)$ for $x<0$
and $h\in\wt\cG$; it follows from \eqref{weylw} that
\begin{equation}\label{weylw4}
  \left( \wt M(x)  \begin{pmatrix} h_{1} \\ h_{2} \end{pmatrix},
  \begin{pmatrix} h_{1} \\ h_{2} \end{pmatrix} \right)
   =  -\dfrac{1}{x} \,(h_{1}, h_{1})+ x \,(h_{2}, h_{2})  
    \leq \left( -\dfrac{M^2}{x}+x\right) \|h_2\|^2.
\end{equation}
Taking $x\leq -M^2$ one has $0< -\frac{M^2}{x}\leq 1$.
Next observe that
\begin{equation}\label{weylw4B}
 \|h\|^2=\|h_1\|^2+\|h_2\|^2\leq (M^2+1)\|h_2\|^2 \quad\Leftrightarrow\quad\|h_2\|^2\geq \dfrac{\|h\|^2}{M^2+1}.
\end{equation}
Now for all $x<\min\,\{-1,-M^2\}$ one has $-\frac{M^2}{x}+x\leq 1+x<0$ and \eqref{weylw4}, \eqref{weylw4B} give the estimate
\[
  \left( \wt M(x)  \begin{pmatrix} h_{1} \\ h_{2} \end{pmatrix},
  \begin{pmatrix} h_{1} \\ h_{2} \end{pmatrix} \right)
  \leq \left(1+x\right) \|h_2\|^2
  \leq \dfrac{1+x}{M^2+1}\, \|h\|^2.
\]
Now assume that $\Theta$ is semibounded from below
with lower bound $\gamma\in\dR$.  Then observe that
\[
 x<\min\,\{-1,-M^2\} \quad \mbox{and} \quad \frac{1+x}{M^2+1}< \gamma
 \quad \Rightarrow \quad \wt M(x)\leq \gamma\, I \leq \Theta,
\]
 which, according to \eqref{semieq},
leads to $x \leq \wt A_\Theta$. Thus, the selfadjoint extension
$\wt A_\Theta$ is bounded from below and this proves the statement.

(ii) Assume that $(R^*)_{\rm s}$ is an unbounded operator.
 Then for each $n\in\dN$ there exist nontrivial
elements  $\{h_{2,n},h_{1,n}\}\in -(R^*)_{\rm s}$
such that $\|h_{1,n}\| \geq  c_n \|h_{2,n}\|$, where $c_n\geq n$.
Now it follows from \eqref{weylw} that for all $x<0$,
\[
\begin{split}
 \left( \wt M(x)  \begin{pmatrix} h_{1,n} \\ h_{2,n} \end{pmatrix},
  \begin{pmatrix} h_{1,n} \\ h_{2,n} \end{pmatrix} \right)
  & =  -\dfrac{1}{x} \,(h_{1,n}, h_{1,n})+ x \,(h_{2,n}, h_{2,n}) \\
  & \geq  \left( -\dfrac{c_n^2}{x}+x\right) \|h_{2,n}\|^2.
  \end{split}
\]
Let $x<0$ be fixed and select $n>|x|$.
Then $-\frac{c_n^2}{x}+x>0$
and thus for every $x<0$ there exists a nontrivial element $h\in\wt\cG$
such that $(\wt M(x)h,h)>0$.
Consider a bounded selfadjoint operator $\Theta$
in $\wt\cG$ and assume that $\wt A_\Theta$
has a lower bound $x<0$.
Combining the previous reasoning with \eqref{semieq}
shows that for some $h\in\wt\cG$
\begin{equation}\label{notdefneg}
  (\Theta h,h)\geq (\wt M(x)h,h)> 0.
\end{equation}
Now take $\Theta=-\delta I_{\wt\cG}$ with $\delta>0$.
Since $\Theta$ is a negative definitive
operator in $\wt\cG$ one concludes from \eqref{notdefneg}
that the corresponding selfadjoint extension
$\wt A_\Theta$ cannot be semibounded from below.
Moreover, here $\|\Theta\|=\delta$ can be made arbitrary small.
This completes the proof.
\end{proof}

The alternative in Theorem \ref{propsemibd} can be stated
 in terms of $R$, instead of its adjoint,
since $(R^*)_{\rm s}$ is a bounded operator precisely
when $\dom R^*$ is closed, which is equivalent to $\dom R$ being closed.
Thus, the operator part $(R^*)_{\rm s}$ of $R^*$ is a bounded
(unbounded) operator if and only if the operator part
$R_{\rm s}$ of $R$ is a bounded (unbounded) operator.
The above proof shows that in case (i)
the upper bound of $\wt M(x)$ tends to $-\infty$ as $x\downarrow -\infty$,
or, in the terminology of \cite{DM1991,DM1995},
$\wt M(x)$ tends uniformly to $-\infty$,
which is the criterion proved therein for the equivalence:
$\Theta$ is semibounded $\Leftrightarrow$ $\wt A_\Theta$ is semibounded.
It is clear from the proof of (ii) that the upper bound, say $\nu_x$,
of $\wt M(x)$ satisfies $\nu_x > 0$,
while $\wt M(x)$ has the weak limit property in \eqref{weylw2}.\\

It is also possible to describe all nonnegative extensions
of the symmetric extension $\wt S$ of $S$ by
a   treatment similar to the one  in Section \ref{Fried}.
 It follows from \eqref{friedforw} and \eqref{kreinforw}
that $\wt S_0=\wt S_F \cap \wt S_K$ is given by
\[
 \wt S_0
 =( \cdom R \oplus \{0\} )^\top
 \times  ( \mul R^* \oplus \cran R )^\top,
\]
and its adjoint is given by
\[
(\wt S_0)^*= ( \cdom R \oplus \ker R^*)^\top
 \times  ( \mul  R^* \oplus \sH_2 )^\top.
\]
One sees immediately that for all $\lambda \in \dC$
\[
\wt \cG_0=\sN_{\lambda}((\wt S_0)^*)= (\{0\} \oplus \ker R^*)^\top
  \subset \left\{ \begin{pmatrix} h_{1}  \\ h_{2} \end{pmatrix} :
 \,\{h_{2},- h_{1}\} \in (R^*)_{\rm s} \,\right\}=\wt \cG.
\]
The details are left to the reader.

\end{document}